\documentclass[english]{smfart}
\usepackage{amsmath, amssymb, amsthm, amscd,smfenum,xspace,mathrsfs,url,upref,verbatim}
\usepackage{dsfont}
\usepackage[all,cmtip]{xy}
\usepackage[utf8]{inputenc}
\usepackage{mathrsfs}
\usepackage{stmaryrd}
\usepackage{bbm}
\usepackage{oldgerm}
\usepackage[english]{babel}
\usepackage[T1]{fontenc}
\usepackage{graphicx}
\usepackage{mathabx}
\usepackage[colorlinks,linkcolor=red,anchorcolor=red,citecolor=blue]{hyperref}

\newtheorem{theorem}[subsection]{Theorem}
\newtheorem{proposition}[subsection]{Proposition}
\newtheorem{corollary}[subsection]{Corollary}
\newtheorem{lemma}[subsection]{Lemma}
\newtheorem{conjecture}[subsection]{Conjecture}

\theoremstyle{definition}

\newtheorem{definition}[subsection]{Definition}
\newtheorem{remark}[subsection]{Remark}

\newtheorem{example}[subsection]{Example}

\numberwithin{equation}{subsection}

\newcommand{\ol}{\overline}

\newcommand{\wt}{\widetilde}

\newcommand{\spec}{\mathrm{Spec}}

\newcommand{\gal}{\mathrm{Gal}}
\newcommand{\ch}{\mathrm{char}}
\newcommand{\cl}{\mathrm{cl}}
\newcommand{\fm}{\mathfrak m}

\newcommand{\pr}{\mathrm{pr}}
\newcommand{\id}{\mathrm{id}}

\newcommand{\dt}{\mathrm{dimtot}}
\newcommand{\rk}{\mathrm{rank}}

\newcommand{\sw}{\mathrm{sw}}

\newcommand{\rt}{\mathrm{t}}

\newcommand{\red}{\mathrm{red}}
\newcommand{\SW}{\mathrm{SW}}

\newcommand{\sO}{\mathscr{O}}

\newcommand{\bZ}{\mathbb Z}
\newcommand{\sF}{\mathscr F}
\newcommand{\bA}{\mathbb A}
\newcommand{\bQ}{\mathbb Q}
\newcommand{\bT}{\mathbb T}

\newcommand{\bF}{\mathbb F}
\newcommand{\sG}{\mathscr G}

\newcommand{\sN}{\mathscr{N}}

\newcommand{\sK}{\mathscr K}

\newcommand{\cA}{\mathcal A}

\DeclareMathOperator{\Sw}{Sw}
\DeclareMathOperator{\Max}{Max}
\DeclareMathOperator{\Sl}{Sl}
\DeclareMathOperator{\rank}{rank}

\DeclareMathOperator{\DT}{DT}

\newcommand{\aut}{\mathrm{Aut}}

\usepackage{xcolor}
\usepackage[english]{babel}

\title[Characteristic cycle and ramification]{Characteristic cycle and wild ramification for nearby cycles of \'etale sheaves}

\usepackage[english]{babel}

\begin{document}

\author[H.Hu]{Haoyu Hu}
\address{Department of Mathematics, Nanjing University, Hankou Road 22, Nanjing 210000, China}
\email{huhaoyu@nju.edu.cn, huhaoyu1987@gmail.com}

\author[J.-B. Teyssier]{Jean-Baptiste Teyssier}
\address{Institut de Math\'ematiques de Jussieu, 4 place Jussieu, Paris, France}
\email{jean-baptiste.teyssier@imj-prg.fr}

\begin{abstract}
In this article, we give a bound for the wild ramification of the monodromy action on the nearby cycles complex of a locally constant \'etale sheaf on the generic fiber of a smooth scheme over an equal characteristic trait  in terms of Abbes and Saito's logarithmic ramification filtration. This provides a positive answer to the main conjecture in \cite{Leal} for smooth morphisms in equal characteristic. We also study the ramification along vertical divisors of \'etale sheaves on relative curves and abelian schemes over a trait.\\


\end{abstract}
\maketitle

The main topic of this article is the wild ramification of the monodromy action on the nearby cycles complex of an \'etale locally constant sheaf defined on the generic fiber of a smooth family of schemes over a trait.

Let $S$ be an henselian trait with a perfect residue field of characteristic $p>0$. Let $s$ be the closed point of $S$. Let $\bar s$ be a geometric closed point of $S$. Let $\eta$ be the generic point of $S$. Let $\bar\eta$ be a geometric generic point of $S$. Let $G$ be the Galois group of $\bar\eta$ over $\eta$. Let $X$ be a scheme and let $f:X\longrightarrow
S$ be a morphism of finite type.  Let $\ell\neq p$  be a prime number. Let $\Lambda$ be a finite field of characteristic $\ell$. Let $\sK
$ be a bounded below complex of \'etale sheaves of $\Lambda$-modules on $X$. Introduced by Grothendieck in the 1960s, the nearby cycles complex $R\Psi(\sK,f)$ and the vanishing cycles
complex $R\Phi(\sK,f)$ of $\sK$ relative to $f$ are complexes  on the geometric special fiber $X_{\bar s}$ of $f$ endowed with a $G$-action, called the \textit{monodromy action}.

The monodromy action on cohomological objects (such as nearby cycles complex, vanishing cycles complex as well as \'etale cohomological groups of the geometric generic fiber) is a central theme in arithmetic geometry. Major contributions include Grothendieck's proof of the local monodromy theorem \cite[XI]{SGA7}, Deligne's Milnor formula computing the total dimension of the vanishing cycles of the constant sheaf at an isolated critical point of a morphism to a curve \cite[XVI]{SGA7}, good and semi-stable reduction criteria for abelian varieties \cite{ST} \cite[IX]{SGA7}, and Saito's proof of the semi-stable reduction theorem for curves \cite{saito87}. Under the semi-stable condition, the tameness of the monodromy action on the nearby cycles complex of the constant sheaf was investigated in \cite{RZ} and \cite{illdwork}.  See also \cite{IllusieMonodromyLocal} and \cite{IllusieNearbyCycles} for a survey. For arbitrary schemes over a trait, the wild ramification of the monodromy action is involved. Bloch's conductor formula \cite{bloch} provides a geometric interpretation for the alternating sum of the Swan conductors of the cohomology groups of the constant sheaf. In this direction, progress to prove this formula were made (among many other works) in  \cite{Abbes} and \cite{KS05}, and generalizations to arbitrary \'etale sheaves were obtained for exemple in \cite{KS13} and \cite{SaitoDirectimage}.

However, very little is known on the wild ramification for the monodromy action on each \textit{individual} cohomology group of an \textit{arbitrary} $\ell$-adic sheaf. In this direction, the following conjecture was formulated in \cite{Leal}.

\begin{conjecture}\label{conjLeal}
 Let $(X,Z)$ be a proper semi-stable pair\footnote{See definition \ref{defsemistable} for details.} over $S$. Let $\sF$ be a locally constant and
constructible sheaf of $\Lambda$-modules on $U:=X\setminus Z$. Suppose that $\sF$ is tamely ramified  along the horizontal part of $Z$. Let $r_{\log}$ be the maximum of the set of Abbes and Saito's logarithmic slopes of $\sF$ at the generic points of the special fiber of $X$.

Then, for every $r>r_{\log}$, the $r$-th upper numbering ramification subgroup of $G$ acts trivially on $H^i_c(U_{\overline{\eta}},\sF|_{U_{\overline{\eta}}})$ for every $i\in \mathbb{Z}_{\geq 0}$.
\end{conjecture}
This conjecture was proved by Leal in \cite{Leal} under the assumption that $S$ is equal characteristic, that $f:X\longrightarrow S$ has relative dimension $1$ and that $\sF$ has rank $1$. The main goal of this article is to prove a local version of this conjecture  for arbitration relative dimension and arbitrary rank sheaf in the good reduction case. The main result is the following (Theorem \ref{genleal}):

\begin{theorem}\label{introgenleal}
Suppose that $S$ is the henselization at a closed point of a smooth curve over a perfect field of characteristic $p>0$. Let $(X,Z)$ be a semi-stable pair over $S$ such that $f : X\longrightarrow S$ is smooth. Let $U:=X\setminus Z$ and $j:U\longrightarrow X$ the canonical injection.
Let $\sF$ be a locally constant and
constructible sheaf of $\Lambda$-modules on $U$. Suppose that $\sF$ is tamely ramified  along the horizontal part of $Z$. Let $r_{\log}$ be the maximum of the set of Abbes and Saito's logarithmic slopes of $\sF$ at the generic points of the special fiber of $X$.

Then, for every $r>r_{\log}$, the $r$-th upper numbering ramification subgroup of $G$ acts trivially on $R^i\Psi(j_{!}\sF,f)$ for every $i\in \mathbb Z_{\geq 0}$.
\end{theorem}
Theorem \ref{introgenleal} says in particular that the slopes of the nearby cycles complex can be bounded in a way depending only on $\sF$ but not on the smooth morphism $f$. This is precisely the content of the boundedness question asked in Question $2$ from \cite{tey}. Due to the $G$-equivariant spectral sequence
\begin{equation*}
E^{ab}_2=H^a(X_s, R^b\Psi(j_!\sF,f))\Longrightarrow H^{a+b}_c(U_{\bar\eta},\sF|_{U_{\bar\eta}}).
\end{equation*}
Theorem \ref{introgenleal} implies that Conjecture \ref{conjLeal} is valid when $f:X\longrightarrow S$ is smooth and $S$ is equal characteristic.

Let us now comment on our approach to Conjecture \ref{conjLeal}. Leal's proof is global. It relies on Kato and Saito's conductor formula \cite{KS13} and Kato's formula \cite{Kato} for the Swan class of a clean rank $1$ sheaf. By contrast, our proof is local. The tools are the \textit{singular support} and the \textit{characteristic cycle} for constructible \'etale sheaves constructed by Beilinson  \cite{bei} and Saito \cite{cc}, respectively, as well as
the semi-continuity properties for Abbes and Saito's ramification invariants following \cite{HY17,Hu}.  The strategy to prove Theorem \ref{introgenleal} has three steps. The first one is to reduce the question to the vanishing of the tame nearby cycles of $j_!(\sF\otimes_{\Lambda}f^*\sN)$ for every isoclinic sheaf of $\Lambda$-modules $\sN$ on $\eta$ with slope strictly bigger than $r_{\log}$. From this point on, we descend to the case where $S$ is a smooth curve. In the second step, when $Z$ is the special fiber of $f$, we prove that for $\sN$ as above, the singular support of $j_!(\sF\otimes_{\Lambda}f^*\sN)$ is supported on the zero section of $\bT^*X$ and the conormal bundle of $Z$ in $X$. In particular, any smooth curve on $X$ transverse to $Z$ is $SS(j_!(\sF\otimes_{\Lambda}f^*\sN))$-transversal. This allows us to reduce the vanishing of the above tame nearby cycles to the curve case, where it is obvious.
The semi-continuity of the largest slopes is the key point in this step. In the last step, we apply the semi-continuity of the largest logarithmic slopes to deal with the general case where $Z$ has some horizontal components.\\ \indent

To describe an application of the above result, we modify our notations. Let $S$ be a smooth curve over a perfect field of characteristic $p>0$. Let $f:X\longrightarrow S$ be a proper and smooth morphism with geometric connected fibers. Let $s$ be a closed point of $S$, $U$ the complement of the fiber $X_s$ in $X$ and $j:U\longrightarrow X$ the canonical injection. Let $\sF$ be a locally constant and constructible sheaf of $\Lambda$-modules on $U$. Using the main Theorem \ref{introgenleal}, we obtain the following boundedness result (Theorem \ref{bounded}):

\begin{corollary}\label{boundedintroduction}
Assume that $f:X\longrightarrow S$ has relative dimension $1$ and genus $g>1$. Let $Z(\sF)$ be the set of closed points of $X_s$ such that $\mathbb T^*_xX$ is contained in $SS(j_!\sF)$. Then, we have
\begin{equation}
\sharp Z(\sF) \leq (2g-2)\cdot \rk_{\Lambda}\sF+2g\cdot \rk_{\Lambda}\sF\cdot  r_{\log}(X_s,\sF),
\end{equation}
where $r_{\log}(X_s,\sF)$ denotes the largest Abbes-Saito's logarithmic slope of $\sF$ at the generic point of $X_s$.
\end{corollary}

If $g=1$, we obtain $Z(\sF)=\emptyset$ by the following theorem due to T. Saito, which proves the non-degeneracy of the ramification of $\sF$ along $X_s$ when $f:X\longrightarrow S$ is an abelian scheme (Theorem \ref{mainprop}):

\begin{theorem}[T. Saito]\label{mainpropintro}
Assume that $X$ is an abelian scheme over $S$. Then, the ramification of $\sF$ along the divisor $X_s$ of $X$ is non-degenerate. In particular, for any closed point $x\in X_s$, the fiber of the singular support  $SS(j_!\sF)\longrightarrow X$ above $x$ is a union of lines in $\mathbb \bT^*_x X$.
\end{theorem}
T. Saito's key idea is the following: a locally constant and constructible \'etale sheaf on an abelian variety over a  field is invariant by translation by some non trivial torsion points depending on the sheaf (proposition \ref{invariant}). When the base field is a strict henselian discrete valuation field, these good torsion points are dense. Thus, we get the non-degeneracy of the ramification of $\sF$ along $X_s$ from the  invariance of $\sF$ by sufficiently many translations. Theorem \ref{mainpropintro} bears some similarity with  Tsuzuki's result on the constancy of Newton polygons for convergent $F$-isocrystals on abelian varieties over a finite field \cite{Tsuzuki}.\\ \indent

The article is organized as follows. In section 1 and 2, we briefly recall Abbes and Saito's ramification theory and introduce the singular support and the characteristic cycle for \'etale sheaves. In section 3, we prove semi-continuity properties for the largest slope and the largest logarithmic slopes for locally constant sheaves ramified along smooth divisors. Relying on these semi-continuity properties, we compute in section 4 the characteristic cycle of a locally constant sheaf on the complement of a smooth divisor after a wild enough twist by a locally constant sheaf coming from the base curve. Based on results in section 3 and section 4, the main Theorem \ref{introgenleal} is proved in section 5. In section 6, we prove corollary \ref{boundedintroduction} by applying our main Theorem \ref{introgenleal}. Following a suggestion from T. Saito, we prove in the last section Theorem \ref{mainpropintro} and other various properties for the wild ramification of $\ell$-adic lisse sheaves on abelian schemes over curves in positive characteristic ramified along a vertical divisor.

\subsection*{Acknowledgement}
We would like to express our gratitude to A. Abbes, A. Beilinson, L. Fu, O. Gabber, F. Orgogozo, T. Saito and Y. Tian for inspiring discussions on this topic and valuable suggestions. We are indebted to T. Saito for providing a simplified argument in the proof of Theorem \ref{sameCC}, as well as for suggesting proposition \ref{invariant}, which gave birth to the last section of this paper. We thank the anonymous referee for careful reading and valuable comments. The first author would like to thank Y. Cao and T. Deng for helpful discussions on abelian varieties.  This manuscript was mainly written while the first author was visiting the Max-Planck Institute for Mathematics in Bonn and while the second author was visiting the Catholic University in Leuven and received support from the long term structural funding-Methusalem grant of the Flemish Government. We would like to thank both Institutes for their hospitality and for providing outstanding working conditions.
The first author is currently supported by the National Natural Science Foundation of China (No. 11901287), the Natural Science Foundation of Jiangsu Province (No. BK20190288) and the Nanjing Science and Technology Innovation Project.

\section{Recollection on the ramification theory of local fields}

\subsection{}
Let $K$ be a henselian discrete valuation field, $\sO_K$ the ring of integer of $K$, $\fm_K$ the maximal ideal of $\sO_K$, $F$ the residue field of $\sO_K$, $\overline K$ a separable closure of $K$, and let $G_K$ be the Galois group of $\overline K$ over $K$.
 We assume that the characteristic of $F$ is $p>0$.
\subsection{}
Suppose that $F$ is perfect. Let $\{G^r_{K,\cl}\}_{r\in \bQ\geq 0}$ be the classical upper numbering ramification filtration on $G_K$  \cite{CL}. For $r\in \bQ_{\geq 0}$, put
\begin{equation*}
G^{r+}_{K,\cl}=\overline{ \bigcup_{s>r}G^s_{K,\cl} }
\end{equation*}
The subgroup $G^0_{K,\cl}$ is the inertia subgroup $I_K$ of $G_K$ and $G^{0+}_{K,\cl}$ is the wild inertia subgroup $P_K$ of $ G_K$.

\subsection{}
 When $F$ is not assumed to be perfect, Abbes and Saito defined in \cite{RamImperfect} two decreasing filtrations $\{G^r_K\}_{r\in\bQ_{>0}}$ and $\{G^r_{K,\log}\}_{r\in\bQ_{\geq 0}}$ on $G_K$ by closed normal subgroups. These filtrations are called respectively {\it the ramification filtration} and {\it the logarithmic ramification filtration}.  For $r\in\bQ_{\geq 0}$, put
\begin{equation*}
G^{r+}_K=\overline{\bigcup_{s>r}G_K^s}\ \ \ \textrm{and}\ \ \ G^{r+}_{K,\log}=\overline{\bigcup_{s>r}G_{K,\log}^s}.
\end{equation*}
We denote by $G^0_K$ the group $G_K$. The ramification filtrations satisfy the following properties:

\begin{proposition}[\cite{RamImperfect,as ii,logcc,wr}]\label{propramfil}
\begin{itemize}
\item[(i)]
For any $0<r\leq 1$, we have
$$G^r_K=G^0_{K,\log}=I_K \text{ and } G^{+1}_K=G^{0+}_{K,\log}=P_K.$$
\item[(ii)]
For any $r\in \bQ_{\geq 0}$, we have
$$G^{r+1}_K\subseteq G^r_{K,\log}\subseteq G^r_K.$$
If $F$ is perfect, then for any $r\in \bQ_{\geq 0}$, we have $$G^r_{K,\cl}=G^r_{K,\log}=G^{r+1}_K.$$
\item[(iii)]
Let $K'$ be a finite extension of $K$ contained in $\overline K$. Let $e$ be the ramification index of $K'/K$. Then, for any $r\in\bQ_{> 1}$, we have $G^{er}_{K'}\subseteq G^r_K$ with equality if $K'/K$ is unramified. For any $r\in\bQ_{>0}$, we have $G^{er}_{K',\log}\subseteq G^r_{K,\log}$ with equality if $K'/K$ is tamely ramified.
\item[(iv)]
For any $r\in \bQ_{>0}$, the graded piece $G^{r}_{K,\log}/G^{r+}_{K,\log}$ is abelian, $p$-torsion and contained in the center of $P_K/G^{r+}_{K,\log}$.  For any $r\in\bQ_{>1}$, the graded piece $G^{r}_{K}/G^{r+}_{K}$ is abelian and contained in the center of $P_K/G^{r+}_{K}$. If we further assume that the characteristic of $K$ is $p>0$, then, for any $r\in\bQ_{>1}$, the graded piece $G^{r}_{K}/G^{r+}_{K}$ is $p$-torsion.
\end{itemize}
\end{proposition}

\subsection{}\label{invM}
Let $\Lambda$ be a finite field of characteristic $\ell\neq p$.
 Let $M$ be a finitely generated $\Lambda$-module with a continuous $P_K$-action. The module $M$ has decompositions
 \begin{equation}\label{twodecomp}
M=\bigoplus_{r\geq 1}M^{(r)}\ \ \ \textrm{and}\ \ \ M=\bigoplus_{r\geq 0}M_{\log}^{(r)}
\end{equation}
into $P_K$-stable submodules, where $M^{(1)}=M^{(0)}_{\log}=M^{P_K}$, and such that for every $r\in \bQ_{>0}$,
\begin{align*}
(M^{(r+1)})^{G^{r+1}_K}=0\ \ \ \textrm{and}\ \ \ (M^{(r+1)})^{G^{(r+1)+}_K}=M^{(r+1)};\\
(M^{(r)}_{\log})^{G^{r}_{K,\log}}=0\ \ \ \textrm{and}\ \ \ (M^{(r)}_{\log})^{G^{r+}_{K,\log}}=M^{r}_{\log}.
\end{align*}
The decompositions \eqref{twodecomp} are called respectively the {\it slope decomposition} and the {\it logarithmic slope decomposition} of $M$. The values $r$ for which $M^{(r)}\neq 0$ (resp. $M^{(r)}_{\log}\neq 0$) are the {\it slopes} (resp. the {\it logarithmic slopes}) of $M$.
Let $\Sl_K(M)$ be the set of slopes of $M$. Let $ \Sl_{K,\log}(M)$ be the set of logarithmic slopes of $M$. For $M\neq 0$, let $r_K(M)$ be the largest slope of $M$.  Let  $r_{K,\log}(M)$ be the largest logarithmic slope of $M$. We say that $M$ is {\it isoclinic} (resp. {\it logarithmic isoclinic}) if $M$ has only one slope (resp. only one logarithmic slope). We say that $M$ is {\it tame} if the action of $P_K$ on $M$ is trivial, that is $M=M^{(1)}=M^{(0)}_{\log}$. By proposition \ref{propramfil} (ii), we have
\begin{equation}\label{inequalityLogNonLog}
r_{K,\log}(M)+1\geq r_K(M)\geq r_{K,\log}(M).
\end{equation}
The {\it total dimension} of $M$ is defined by
\begin{equation}\label{dimtot}
\dt_K(M):=\sum_{r\geq 1}r\cdot\dim_{\Lambda} M^{(r)}.
\end{equation}
The {\it Swan conductor} of $M$ is defined by
\begin{equation*}
\sw_K(M):=\sum_{r\geq 0}r\cdot\dim_{\Lambda} M^{(r)}_{\log}
\end{equation*}
We have
\begin{equation*}
\sw_K(M)+\dim_{\Lambda}M\geq \dt_K(M)\geq \sw_K(M).
\end{equation*}
If the residue field $F$ is perfect, we have
\begin{align}
r_{K,\log}(M)+1&=r_K(M)\\
\sw_K(M)+\dim_{\Lambda}M&= \dt_K(M)
\end{align}
and $\sw_K(M)$ is the classical Swan conductor of $M$.

Let $1\leq r_1<r_2<\cdots<r_n$ be all slopes of $M$, the {\it Newton polygon} of the slope decomposition  of $M$ denotes the lower convex hull in $\mathbb R^2$ of the following points:
$$
(0,0)\ \ \textrm{and}\ \ \left(\sum_{i=1}^j\dim_\Lambda M^{(r_i)}, \sum_{i=1}^j r_i\cdot\dim_\Lambda M^{(r_i)}\right)\ \  (j=1,\dots, n).
$$

Let $K'$ be a finite separable extension of $K$ contained in $\overline K$. Let $e$ be the ramification index of $K'/K$. By proposition \ref{propramfil} (iii), we have
\begin{align*}
e\cdot r_K(M)\geq r_{K'}(M)\ \ \ &\textrm{and}\ \ \ e\cdot r_{K,\log}(M)\geq r_{K',\log}(M);\\
e\cdot\dt_K (M)\geq\dt_{K'}(M) \ \ \ &\textrm{and}\ \ \ e\cdot\sw_K (M)\geq\sw_{K'}(M).
\end{align*}
If $K'/K$ is tamely ramified, we have
\begin{equation*}
e\cdot r_{K,\log}(M)= r_{K',\log}(M) \ \ \ \textrm{and}\ \ \ e\cdot\sw_K (M)=\sw_{K'}(M).
\end{equation*}
If $K'/K$ is unramified, we have
\begin{equation*}
r_K(M)=r_{K'}(M)\ \ \ \textrm{and}\ \ \ \dt_K (M)=\dt_{K'}(M)
\end{equation*}

\subsection{}
We assume that the characteristic of $K$ is $p>0$. Let $M$ be a finitely generated $\Lambda$-module with a continuous $P_K$-action. Suppose that $M$ is isoclinic of slope $r\in \bQ_{>1}$. Then, $M$ has a unique direct sum decomposition
\begin{equation*}
M=\bigoplus_{\chi\in X(r)} M_{\chi}
\end{equation*}
into $P_K$-stable submodules $M_{\chi}$, where $X(r)$ is the set of isomorphism classes of non-trivial finite characters $\chi:G^r_K/G^{r+}_K\longrightarrow \Lambda^{\times}$ and  where $M_{\chi}$ is a direct sum of finitely many copies of $\chi$.

Assume moreover that $F$ is of finite type over a perfect field. We denote by $\sO_{\overline K}$ the integral closure of $\sO_K$ in $\overline K$, by $\overline F$ the residue field of $\sO_{\overline K}$ and by $\mathrm{ord}:\overline K\longrightarrow\bQ\bigcup\{\infty\}$ a valuation normalized by $\mathrm{ord}(K^{\times})=\bZ$. For any $r\in\bQ_{\geq 0}$, we put
\begin{equation*}
\fm^r_{\overline K}=\{x\in\overline K^{\times}\,;\,\mathrm{ord}(x)\geq r\}\ \ \ \textrm{and}\ \ \ \fm^{r+}_{\overline K}=\{x\in\overline K^{\times}\,;\,\mathrm{ord}(x)>r\}.
\end{equation*}
The quotient $\fm^r_{\overline K}/\fm^{r+}_{\overline K}$ is a $1$-dimensional $\overline F$-vector space. For $r\in\bQ_{>1}$, there exists an injective homomorphism called the {\it characteristic form}
\begin{equation}\label{charform}
\xymatrix{
\ch:\mathrm{Hom}_{\mathbb F_p}(G^r_K/G^{r+}_K,\mathbb F_p)\ar[r]& \mathrm{Hom}_{\overline F}(\fm^r_{\overline K}/\fm^{r+}_{\overline K},\Omega^1_{\sO_K}\otimes_{\sO_K}\overline F).
}
\end{equation}

\subsection{}
We present three elementary lemmas on Galois modules that will be used in the sequel of the article.

\begin{lemma}\label{Nwithprescribedslope}
Assume that $\ch(K)=p$, that the residue field $F$ is perfect and that
$\Lambda$  contains a primitive $p$-th root of unity. Then, for any two positive integers $r,s$ co-prime to $p$, there exists a non-trivial $\Lambda$-module $N$ with continuous $G_K$-action such that $N$ is pure of logarithmic slope $r/s$.
\end{lemma}

\begin{proof}
Let $\pi$ be a uniformizer of $K$, let $L_1=K[\theta]/(\theta^s-\pi)$ and $L_2=L_1[t]/(t^p-t-\theta^{-r})$. The extension $L_2/L_1$ is Galois with $\gal(L_2/L_1)\cong \bZ/p\bZ$. Let $\chi:\gal(L_2/L_1)\longrightarrow \Lambda^{\times}$ be a non-trivial representation and we denote by $N_1$ the associated $\Lambda$-module of rank $1$ with a continuous $G_{L_1}$-action.
Since $(r,p)=1$, we have $r_{L_1,\log}(N_1)=r$ (\cite[1.1.7]{lau}). Since $L_1/K$ is tamely ramified with $e(L_1/K)=s$, the $G_{L_1}$-module $$\mathrm{Res}_{G_{K}}^{G_{L_1}}\mathrm{Ind}_{G_{K}}^{G_{L_1}}N_1
\cong N_1\otimes_{\Lambda}\mathrm{Res}_{G_{K}}^{G_{L_1}}
\mathrm{Ind}_{G_{K}}^{G_{L_1}}\Lambda$$
is pure of logarithmic  slope $r$. Hence the $G_K$-module $N=\mathrm{Ind}_{G_{K}}^{G_{L_1}}N_1$ is logarithmic isoclinic of slope $r/s$.
\end{proof}

\begin{lemma}\label{tensordtsw}
Let $M$ and $N$ be finitely generated $\Lambda$-modules with continuous $G_K$-actions.
Put $r=r_K(M)$ and $s=r_K(N)$.
Put $r_{\log}=r_{K,\log}(M)$ and $s_{\log}=r_{K,\log}(N)$.
\begin{itemize}
\item[(i)]
If $r\geq s$, then we have $r_K(M\otimes_\Lambda N)\leq r$ and
\begin{equation}\label{dimtotless}
\dt_K(M\otimes_\Lambda N)\leq\dim_{\Lambda} N\cdot\dim_{\Lambda} M\cdot r
\end{equation}
Assume that $N$ is isoclinic and that $s$ is not a slope of $M$.
Assume that $M$ and $N$ are not trivial.
Then,
\begin{equation}\label{dimtotbigger}
\dt_K(M\otimes_\Lambda N) \geq \dim_{\Lambda} N\cdot\dim_{\Lambda} M\cdot s
\end{equation}
with equality if and only if $s>r$.
In the latter case, $M\otimes_\Lambda N$ is isoclinic of slope $s$.\\
\item[(ii)]
If $r_{\log}\geq s_{\log}$,  then we have $r_{K,\log}(M\otimes_\Lambda N)\leq r_{\log}$ and
\begin{equation*}
\sw_K(M\otimes_\Lambda N)\leq\dim_{\Lambda} N\cdot\dim_{\Lambda} M\cdot r_{\log}.
\end{equation*}
Assume that $N$ is logarithmic isoclinic and that $s_{\log}$ is not a logarithmic slope of $M$.
Assume that $M$ and $N$ are not trivial.
Then,
\begin{equation*}
\sw_K(M\otimes_\Lambda N) \geq \dim_{\Lambda} N\cdot\dim_{\Lambda} M\cdot s_{\log}
\end{equation*}
with equality if and only if $s_{\log}>r_{\log}$.
In the latter case, $M\otimes_\Lambda N$ is logarithmic isoclinic of slope $s_{\log}$.
\end{itemize}
\end{lemma}
\begin{proof}
The proofs of $(i)$ and $(ii)$ are the same. We prove $(i)$. Since the action of $G^{r+}_K$ on $M$ and $N$ is trivial, its action on $M\otimes_{\Lambda} N$ is trivial too. Hence, $r_K(M\otimes_{\Lambda} N)\leq r$ and we get \eqref{dimtotless}.
Assume that $N$ is isoclinic and that $s$ is not a slope of $M$.
To prove the second part of lemma \ref{tensordtsw} $(i)$, we can suppose that $M$ is isoclinic.
If $r>s$, we have
$M\otimes_{\Lambda} N\cong M^{\oplus\dim_{\Lambda}N}$ as $G^r_K$-representation. Hence $(M\otimes_{\Lambda} N)^{G^r_K}=(M^{G^r_K})^{\oplus\dim_{\Lambda}N}=\{0\}$. Hence $M\otimes_{\Lambda} N$ is isoclinic of slope $r$.
Thus,
$$
\dt_K(M\otimes_\Lambda N)=\dim_{\Lambda} N\cdot\dim_{\Lambda} M\cdot r > \dim_{\Lambda} N\cdot\dim_{\Lambda} M\cdot s
$$
If $r<s$, we have
$$
\dt_K(M\otimes_\Lambda N)= \dim_{\Lambda} N\cdot\dim_{\Lambda} M\cdot s
$$
and the second part of lemma \ref{tensordtsw} $(i)$ is proved.
\end{proof}

\begin{lemma}\label{added_referee_lemma}
Let $M$ be a non trivial finitely generated $\Lambda$-module with continuous $G_K$-action.
Let $s\geq 0$ be a real number.
\begin{enumerate}
\item If $N$ is a finitely generated $\Lambda$-module of pure logarithmic slope $s$, and if $s $ is not a logarithmic slope of $M$, then $(M\otimes_{\Lambda} N)^{P_K}=0$.\\
\item If $(M\otimes_{\Lambda} N)^{P_K}=0$ for every  finitely generated $\Lambda$-module of pure logarithmic slope $s$, then $s$ is not a logarithmic slope of $M$.
\end{enumerate}
\end{lemma}
\begin{proof}
$(1)$ We can suppose that $M$ is logarithmic isoclinic and that $N$ is not trivial.
Then, lemma \ref{tensordtsw} $(ii)$ implies that $M\otimes_{\Lambda} N$ is isoclinic of logarithmic slope $\Max(r_{K,\log}(M)),s)>0$, so that $(M\otimes_{\Lambda} N)^{P_K}=0$ follows.\\
$(2)$ Put $N=(M^{(s)}_{\log})^{\vee}$ and observe that $N$ has pure logarithmic slope $s$.
The inclusion $M^{(s)}_{\log}\longrightarrow  M$ induces an inclusion
$$
(M^{(s)}_{\log}\otimes_{\Lambda} N)^{P_K}\longrightarrow (M \otimes_{\Lambda} N)^{P_K}
$$
If $(M \otimes_{\Lambda} N)^{P_K}=0$, we deduce that the identity morphism of $M^{(s)}_{\log}$ is the zero element of the space of endomorphisms of $M^{(s)}_{\log}$.
Hence, $M^{(s)}_{\log}=0$ and the proof of lemma \ref{added_referee_lemma} is complete.
\end{proof}

\section{Singular support and characteristic cycle  for $\ell$-adic sheaves}

\subsection{}
Let $k$ be a field of characteristic $p>0$.  Let $X$ be a smooth connected $k$-scheme. Let $C$ be a closed conical subset of the cotangent bundle $\bT^*X$. Let $x$ be a point of $X$ and $\bar x\longrightarrow X$ a geometric point of $X$ above $x$. Set $\bT^*_{x}X=\bT^*X\times_X x$   and $C_{ x}=C\times_X x$.  Set $\bT^*_{\bar x}X=\bT^*X\times_X\bar x$ and  $C_{\bar x}=C\times_X\bar x$. Let $\Lambda$ be a finite field of characteristic $\ell\neq p$.

\subsection{}
Let $f:U\longrightarrow X$ be a morphism of smooth $k$-schemes. Let $u\in U$ and let $\bar u\longrightarrow U$ be a geometric point of $U$ lying over $u$. We say that $f:U\longrightarrow X$ is $C$-{\it transversal at} $u$ if $\ker df_{\bar u} \bigcap C_{f(\bar u)}\subseteq \{0\}\subseteq
\bT^*_{f(\bar u)}X$, where $df_{\bar u}:\bT^*_{f(\bar u)}X\longrightarrow \bT^*_{\bar u}U$ is the cotangent map of $f$ at $\bar u$. We say that $f:U\longrightarrow Y$ is $C$-transversal if it is $C$-transversal at every point of $U$. For a $C$-transversal morphism $f:U\longrightarrow X$, let $f^\circ C$ be the image of $C\times_XU$ in $\bT^*U$ by  $df:\bT^*X\times_XU\longrightarrow\bT^*U$.
Note that $f^\circ C$ is a closed conical subset of $\bT^*U$.

Let $g:X\longrightarrow Y$ be a morphism of smooth $k$-schemes.  Let $x\in X$ and let $\bar x\longrightarrow X$ be a geometric point of $X$ lying over $x$. We say that $g:X\longrightarrow Y$ is $C$-{\it transversal at} $x$ if $dg_{\bar x}^{-1}(C_{\bar x})=\{0\}\subset\bT^*_{\bar x}Y$. We say that $g:X\longrightarrow Y$ is $C$-{\it transversal } if it is $C$-transversal at every point of $X$.

Let $(g,f):Y\longleftarrow U\longrightarrow X$ be a pair of morphisms of smooth $k$-schemes. We say that $(g,f)$ is $C$-{\it transversal} if $f:U\longrightarrow X$ is $C$-transversal and if $g:U\longrightarrow Y$ is $f^\circ C$-transversal.

Let $\sF$ be an object in $D^b_c(X,\Lambda)$. We say that $\sF$ is {\it micro-supported on} $C$ if for any $C$-transversal pair of morphisms of smooth $k$-schemes $(g,f):Y\longleftarrow U\longrightarrow X$, the map $g:U\longrightarrow Y$ is universally locally acyclic with respect to $f^*\sF$. We say that $\sF$ is {\it weakly micro-supported} on $C$ if for any $C$-transversal pair of morphisms of smooth $k$-schemes $(g,f):\mathbb A^1_k\longleftarrow U\longrightarrow X$ satisfying
\begin{itemize}
\item[(1)]
If $k$ is infinite, the map $f:U\longrightarrow X$ is an open immersion;
\item[(2)]
If $k$ is finite, there exists a finite extension $k'/k$ such that $f:U\longrightarrow X$ is the composition of the first projection  $\pr_1:U\otimes_{k}k'\longrightarrow U$ with an open immersion $h: U\longrightarrow X$;
\end{itemize}
the map $g:U\longrightarrow Y$ is universally locally acyclic with respect to $f^*\sF$.

\vspace{0.2cm}
The following theorem is due to Beilinson \cite[Th. 1.3, Th. 1.5]{bei}.

\begin{theorem}\label{beilinsonmain}
Let $\sF$ be an object in $D^b_c(X,\Lambda)$. Then, there exists a smallest closed conical subset $SS(\sF)$ in $\bT^*X$ on which $\sF$ is micro-supported. The closed conical subset $SS(\sF)$ coincides with the smallest closed conical subset in $\bT^*X$ on which $\sF$ is weakly micro-supported. Furthermore, $SS(\sF)$ is equidimensional of dimension $\dim_k X$.
\end{theorem}
The closed conical subset $SS(\sF)$ is called \textit{the singular support of $\sF$}.

\subsection{}\label{def df}
We assume that $k$ is perfect. Let $S$ be a smooth $k$-curve. Let $f:X\longrightarrow S$ be a morphism. Let $x$ be a closed point of $X$ and put $s=f(x)$. We say that $x$ is an {\it at most $C$-isolated characteristic point} for $f:X\longrightarrow S$ if $f:X\setminus \{x\}\longrightarrow S$ is $C$-transversal. Any local trivialization of
$\bT^*S$ in a neighborhood of $s$ gives rise to a local section of $\bT^*X$ by applying $df:\bT^*S\times_SX\longrightarrow\bT^*X$. We abusively denote by $df$ this section. Let $A$ be a cycle in $\bT^*X$ supported on $C$. Under the condition that $x$ is an at most $C$-isolated characteristic point of $f:X\longrightarrow S$, the intersection of $A$  and the cycle $[df]$ is supported at most at a single point in $\bT^*_xX$. Since $C$ is conical, the intersection number $(A, [df])_{\bT^*X,x}$ is independent of the chosen local trivialization for $\bT^*S$ in a neighborhood of $s$.

The following theorem is due to Saito \cite[Th. 5.9]{cc}.
\begin{theorem}
Let $\sF$ be an object in $D^b_c(X,\Lambda)$.  There exists a unique cycle $CC(\sF)$ in $\bT^{\ast}X$ supported on $SS(\sF)$  such that for every \'etale morphism $h: U\longrightarrow X$, for every morphism $f: U\longrightarrow S$ with $S$ a smooth $k$-curve,  for every at most $h^{\circ}(SS(\sF))$-isolated characteristic point $u\in U$ for $f:U\longrightarrow S$, we have the following Milnor type formula
 \begin{equation}\label{Milnor}
 -\sum_i(-1)^i\dt (R^i\Phi_{\bar u}(h^{\ast}\sF, f))=(h^*(CC(\sF)),[df])_{T^{\ast}U, u},
 \end{equation}
 where $R\Phi_{\bar u}(h^{\ast}\sF, f)$ denotes the stalk of the vanishing cycle of $h^*\sF$ with respect to $f:U\longrightarrow S$ at the geometric point $\bar u\longrightarrow U$ above $u$.
\end{theorem}
The cycle $CC(\sF)$ is called the {\it characteristic cycle} of $\sF$.

\begin{example}\label{CCcurve}
Assume that $X$ is a smooth $k$-curve. Let $U$ be an open dense subscheme of $X$ and $j:U\longrightarrow X$ the canonical injection. Let $\sF$ be a locally constant and constructible sheaf of $\Lambda$-modules on $U$. Then we have
\begin{equation*}
CC(j_!\sF)=-\rk_{\Lambda}\sF\cdot[\bT^*_XX]-\sum_{x\in X\backslash U}\dt_x(\sF)\cdot[\bT^*_xX]
\end{equation*}
\end{example}

The following facts are due to Saito \cite[5.13, 5.14]{cc}.
\begin{proposition}\label{51314}
\begin{itemize}
\item[(i)]
Let $\sF$ be an object in $D^b_c(X,\Lambda)$. If $\sF$ is perverse, then, the support of $CC(\sF)$ is $SS(\sF)$ and the coefficients of $CC(\sF)$ are positive.
\item[(ii)] Let $f:Y\longrightarrow X$ be a separated morphism of smooth $k$-schemes. For every object $\sG$ in $D^b_c(Y,\Lambda)$, we have $CC(Rf_*\sG)=CC(Rf_!\sG)$.
\end{itemize}
\end{proposition}

\subsection{}
Assume that $C$ has pure dimension $\dim_k X$. Let $f:Y\longrightarrow X$ be a $C$-transversal morphism of smooth connected $k$-schemes. We say that $f:Y\longrightarrow X$ is {\it properly $C$-transversal} if every irreducible component of $Y\times_XC$  has dimension $\dim_k Y$. Assume that $f:Y\longrightarrow X$ is properly $C$-transversal. Consider the diagram
\begin{equation*}
\xymatrix{
\bT^*Y       & \ar[l]_-{df}  Y  \times_X \bT^*X   \ar[r]^-{\pr_2}  & \bT^*X
}
\end{equation*}
 Let $A$ be a cycle in $\bT^*X$ supported on $C$. We put
\begin{equation*}
 f^!A := (-1)^{\dim_k X-\dim_k Y} df_{\ast}\pr_2^*A
\end{equation*}
The following compatibility of the characteristic cycle with properly transversal morphism is due to Saito \cite[Th. 7.6]{cc}
\begin{theorem}\label{cc7.6}
Let $\sF$ be an object in $D^b_c(X,\Lambda)$. Let $f:Y\longrightarrow X$ be a morphism of smooth connected $k$-schemes such that $f:Y\longrightarrow X$ is properly $SS(\sF)$-transversal. Then, we have
\begin{equation*}
f^{!}(CC(\sF))=CC(f^*\sF).
\end{equation*}
\end{theorem}
Assume that $C$ has pure dimension $\dim_k X$. Let $f:X\longrightarrow Y$ be a $C$-transversal morphism of smooth connected $k$-schemes with $\dim_k Y \leq \dim_k X$. We say that $f:X\longrightarrow Y$ is {\it properly $C$-transversal} if for every closed point $y$ in $Y$, the fiber $y\times_Y C$  has dimension $\dim_k X-\dim_k Y$. The following theorem is due to Saito \cite[Th. 2.2.3]{SaitoDirectimage}.
\begin{theorem}\label{conductor}
Let $f:X\longrightarrow S$ be a quasi-projective morphism with $S$ a smooth connected $k$-curve. Let $\bar\eta$ be a geometric generic point of $S$. Let $s\in S$ be a closed point. Let $\sF$ be an object in $D^b_c(X,\Lambda)$. Assume that $f:X\longrightarrow S$ is proper on the
support of $\sF$ and properly $SS(\sF)$-transversal over a dense open subset in $S$. Then,
we have
\begin{equation*}
(CC(\sF), [df])_{\bT^*X,X_s}= -\sum_{i\in\mathbb Z}(-1)^i(\dim_{\Lambda}(R^if_*\sF)_{\bar\eta}-\dim_{\Lambda}(R^if_*\sF)_s+\sw_s(R^if_*\sF))
\end{equation*}
\end{theorem}

\subsection{}
Let $f:Y\longrightarrow X$ be a separated morphism of smooth $k$-schemes.
Let $\sF$ be an object in $D^b_c(X,\Lambda)$.
There is a canonical morphism
\begin{equation}\label{cyclemap}
c_{f,\sF}:f^*\sF\otimes_{\Lambda}^L Rf^!\Lambda\longrightarrow Rf^!\sF
\end{equation}
obtained by adjunction from the composition
\begin{equation*}
Rf_!(f^*\sF\otimes^L_{\Lambda}Rf^!\Lambda)\xrightarrow{\sim} \sF\otimes^L_{\Lambda}Rf_!Rf^!\Lambda\longrightarrow\sF
\end{equation*}
where the first arrow is the projection formula and where the second arrow is induced by the adjunction $Rf_!Rf^!\Lambda\longrightarrow \Lambda$.

The following proposition is due to Saito \cite[Prop. 8.13]{cc}.
\begin{proposition}\label{8.13}
Let $\sF$ be an object in $D^b_c(X,\Lambda)$.
Let $f:Y\longrightarrow X$ be a separated morphism of smooth $k$-schemes.
If $f:Y\longrightarrow X$ is $SS(\sF)$-transversal, then the canonical morphism \eqref{cyclemap} is an isomorphism.
\end{proposition}

\begin{corollary}\label{basechange}
Let
\begin{equation*}
\xymatrix{\relax
V\ar[r]^{h'}\ar[d]_{j'}\ar@{}|-{\Box}[rd]&U\ar[d]^j\\
W\ar[r]_h&X}
\end{equation*}
be a cartesian diagram of smooth schemes over $k$ such that the vertical arrows are open immersions and such that $h$ is separated.
Let $\sF$ be an object in $D^b_c(X,\Lambda)$.
Suppose that  $h : W\longrightarrow X$ is $SS(Rj_*\sF)$-transversal.
Then, there is an isomorphism in $D^b_c(X,\Lambda)$
\begin{equation}\label{keybc}
\xymatrix{
h^*Rj_*\sF \ar[r]^-{\sim} &  Rj'_*h'^*\sF.
}
\end{equation}
\end{corollary}
\begin{proof}
Since the schemes $X$ and $W$ are smooth, we have
$Rh^!\Lambda\cong \Lambda(c)[2c]$ and $Rh^{\prime !}\Lambda\cong \Lambda(c)[2c]$ for some integer $c$.
To prove corollary \ref{basechange}, it is thus enough to prove that the base change morphism
\begin{equation}\label{bc1}
\xymatrix{
h^*Rj_*\sF \otimes_{\Lambda}^L Rh^!\Lambda \longrightarrow  Rj'_*(h'^*\sF\otimes_{\Lambda}^L Rh^{\prime !}\Lambda).
}
\end{equation}
is an isomorphism.
Since $h:W\longrightarrow X$ is $SS(Rj_*\sF)$-transversal, proposition \ref{8.13} implies that the canonical morphism
$$
c_{h,Rj_*\sF}:h^*Rj_*\sF\otimes_{\Lambda}^L Rh^!\Lambda \longrightarrow Rh^! Rj_*\sF
$$
is an isomorphism.
On the other hand, the fact that $h:W\longrightarrow X$ is $SS(Rj_*\sF)$-transversal implies that $h':V\longrightarrow U$ is $SS(\sF)$-transversal.
Hence, proposition \ref{8.13} implies that the canonical morphism
$$
c_{h',\sF}:h'^*\sF\otimes_{\Lambda}^L Rh^{\prime !}\Lambda \longrightarrow Rh^{\prime !} \sF
$$
is an isomorphism.
Hence, the base change morphism \eqref{bc1} identifies with
\begin{equation}\label{bc2}
\xymatrix{
Rh^! Rj_*\sF \longrightarrow  Rj'_* Rh^{\prime !} \sF.
}
\end{equation}
By proper base change theorem \cite[XVIII 3.1.13]{SGA4-3}, the morphism \eqref{bc2} is an isomorphism, and proposition \ref{basechange} follows.
\end{proof}

\subsection{}\label{ccnondeg}
Let $D$ be a smooth connected divisor in $X$. Put $U:=X\setminus D$. Let $j:U\longrightarrow X$ be the canonical injection. Let  $\xi$ be the generic point of $D$, $K$ the fraction field of the henselization $X_{(\xi)}$ of $X$ at $\xi$, $\eta=\spec(K)$ the generic point of $X_{(\xi)}$, $\overline K$ a separable closure of $K$, $F$ the residue field of the ring of integer $\sO_K$ of $K$, $\overline F$ the residue field of the ring of integer $\sO_{\overline K}$ of $\overline K$, and let $G_K$ be the Galois group of $\overline K$ over $K$.  Let $\sF$ be a locally constant and constructible sheaf of $\Lambda$-modules on $U$.

We suppose that the ramification of $\sF$ along $D$ is non-degenerate. In a nutshell, this means that \'etale locally along $D$, the sheaf $\sF$  is a direct sum of sheaves of $\Lambda$-modules which are isoclinic at $\xi$ and whose ramification  at a closed point of $D$ is controlled by the ramification at $\xi$. See \cite[Def. 3.1]{wr} for details. Under the non-degenerate condition, the characteristic cycle of $j_!\sF$ can be expressed in terms of ramification theory. Let us explain how.

Let $M$ be the finitely generated $\Lambda$-module  with continuous $G_K$-action corresponding to $\sF|_{\eta}$. Let
 \begin{equation*}
 M=\bigoplus_{r\in \bQ_{\geq 1}}M^{(r)}
 \end{equation*}
 be the slope decomposition of $M$. For $r\in \mathbb{Q}_{>1}$, let
 \begin{equation*}
 M^{(r)}=\bigoplus_{\chi}M^{(r)}_{\chi}
 \end{equation*}
be the central character decomposition of $M^{(r)}$, for $r>1$. Let $\psi$ be a non-trivial character $\psi:\bF_p\longrightarrow \Lambda^{\times}$. For $r\in\mathbb Q_{>1}$, the graded piece $G^r_K/G^{r+}_K$ is abelian and $p$-torsion. Hence, each central character $\chi$ uniquely factors as $\chi:G^r_K/G^{r+}_K\longrightarrow \bF_p\xrightarrow{\psi}\Lambda^{\times}$. We also denote by $\chi$ the induced character $G^r_K/G^{r+}_K\longrightarrow \bF_p$ and by
\begin{equation*}
\ch(\chi):\fm^r_{\overline K}/\fm^{r+}_{\overline K}\longrightarrow\Omega^1_{\sO_K}\otimes_{\sO_K}\overline F
\end{equation*}
the characteristic form of $\chi$. Let $F_\chi$ be a finite extension of $F$ contained in $\overline F$ such that $\ch(\chi)$ is defined over $F_{\chi}$. The characteristic form $\ch(\chi)$ defines a line in $\bT^*X\times_X\spec F_{\chi}$. Let $\overline L_{\chi}$ be the closure of the image of $L_{\chi}$ in $\bT^*X$. If the ramification of $j_!\sF$ is non-degenerate along $D$, we have \cite[Th.~7.14]{cc}:
\begin{align}\label{CCnondeg}
CC(j_!\sF)=(-1)^{\dim_k X} \bigg(&\rk_{\Lambda}\sF\cdot[\bT^*_{X}X]+\dim_{\Lambda}M^{(1)}\cdot[\bT^*_{D}X]\\
&+\sum_{r>1}\sum_{\chi}\frac{r\cdot\dim_{\Lambda}M^{(r)}_{\chi}}{[F_{\chi}:F]}[\overline L_{\chi}]       \bigg)\nonumber.
\end{align}
When $\sF$ is not trivial, the non-degeneracy of the ramification of $\sF$ along $D$ implies that for every geometric point $\bar x\longrightarrow D$, the fiber $SS(j_!\sF)\times_X \bar x$ is a finite union of $1$-dimensional vector spaces in $\bT^*_{\bar x}X$.

In general, the ramification of $\sF$ along $D$ becomes non-degenerate after removing a codimension one closed subset in $D$ (cf. \cite[Lemma 3.2]{wr}).


\section{Semi-continuity for the largest slope}

\subsection{}\label{irrdiv}
Let $k$ be a perfect field of characteristic $p>0$. Let $\Lambda$ be a finite field of characteristic $\ell\neq p$. Let $X$ be a smooth $k$-scheme. Let $D$ be a smooth connected divisor in $X$. Put $U:=X\setminus D$. Let $j:U\longrightarrow X$ be the canonical injection. Let $\xi$ the generic point of $D$, $K$ the fraction field of the henselization $X_{(\xi)}$ of $X$ at $\xi$, $\eta=\spec (K)$ the generic point of $X_{(\xi)}$, $\overline K$ an algebraic closure of $K$, and let $G_K$ be the Galois group of $\overline K$ over $K$.

\subsection{}\label{irrdivsheaf}
Let $\sF$ be a locally constant and constructible sheaf of $\Lambda$-modules on $U$. The restriction $\sF|_{\eta}$ corresponds to a  finitely generated $\Lambda$-module with continuous $G_K$-action. Let $r(D,\sF)$ be the largest slope of $\sF|_{\eta}$. Set $R_D(\sF):=r(D,\sF)\cdot D$ and we call it the {\it largest slope divisor} of $\sF$. Let $r_{\log}(D,\sF)$ be the largest logarithmic slope of $\sF|_{\eta}$. Set $R_{D,\log}(\sF)=r_{\log}(D,\sF)\cdot D$ and we call it the {\it largest logarithmic slope divisor} of $\sF$. Similarly, we denote by $\dt_D(\sF)$ the total dimension of $\sF|_{\eta}$, by
\begin{equation*}
\DT_D(\sF):=\dt_D(\sF)\cdot D
\end{equation*}
the {\it total dimension divisor} of $\sF$, by $\sw_D(\sF)$ the Swan conductor of $\sF|_{\eta}$ and by
\begin{equation*}
\SW_D(\sF):=\sw_D(\sF)\cdot D
\end{equation*}
the {\it Swan divisor} of $\sF$.

In  \cite{HY17,Hu}, the authors studied the behavior of the total dimension divisors and the Swan divisors by pull-back, and deduced  several semi-continuity properties.  In this section, we pursue a similar study for the largest slope divisor and the largest logarithmic slope divisor. For an analogue for differential systems in characteristic $0$, we refer to \cite{Andre}.

\vspace{0.2cm}
In the following of this section, we assume that $k$ is algebraically closed.

\begin{proposition}\label{slope'}
\begin{itemize}
\item[(1)]{ \cite[Prop. 2.22.2]{wr}}
Let $C$ be a smooth $k$-curve in $X$ meeting $D$ transversally at a closed point $x$. We put $C_0=C\backslash \{x\}$.
 Then, if the ramification of $\sF$ along $D$ is non-degenerate at $x$ and the immersion $i:C\longrightarrow X$ is $SS(j_!\sF)$-transversal at $x$, then  \begin{equation*}
r(x,\sF|_{C_0})= r(D,\sF).
\end{equation*}
\item[(2)]
The following two conditions are equivalent:
\begin{itemize}
\item[(a)]
$r\geq r(D,\sF)$; 
\item[(b)]
$r\geq r(x,\sF|_{C_0})$ for any smooth $k$-curve $C$ in $X$ meeting $D$ transversally at a closed point $x$, where $C_0=C\setminus \{x\}$.
\end{itemize}
\end{itemize}
\end{proposition}

\begin{proof}
 Since the slopes are unchanged by scalar extensions, we can assume that $\Lambda$ contains a primitive $p$-th root of unity.

$(1)$ Since the ramification of $\sF$ along $D$ is non-degenerate at $x$, at the cost of replacing $X$ by an \'etale neighborhood of $x$, we can suppose that $\sF$ is a direct sum of locally constant and constructible sheaves $\{\sF_\alpha\}_{\alpha\in I}$ of $\Lambda$-modules on $U$ such that for every $\alpha\in I$, the ramification of $\sF_\alpha$ along $D$ is non-degenerate and  $\sF_\alpha|_{\eta}$ is isoclinic. Since $i:C\longrightarrow X$ is $SS(j_!\sF)$-transversal at $x$, it is $SS(j_!\sF_{\alpha})$-transversal at $x$ for every $\alpha\in I$. Hence, proposition \ref{slope'} $(1)$ is a direct consequence of \cite[Prop. 2.22.2]{wr}.

$(2)$ We firstly show (a)$\Rightarrow $(b). It is sufficient to show that $r(D,\sF)\geq r(x,\sF|_{C_0})$ for a smooth $k$-curve $C$ in $X$ meeting $D$ transversally at a closed point $x$.
At the cost of replacing $X$ by an \'etale neighborhood of $x$, we can suppose that there exists a smooth morphism $f:X\longrightarrow T$ with $T$ a smooth $k$-curve, and a closed point $t\in T$ such that $D=f^{-1}(t)$.
From lemma \ref{Nwithprescribedslope}, we can find a rational number $s$ such that $s>r(D,\sF)$, that $s$ is not a slope of $\sF|_{C_0}$ at $x$ and that there exists a locally constant and constructible sheaf of $\Lambda$-modules $\sN$ on the generic point of $T_{(t)}$ such that $\sN$ is isoclinic with slope $s$.
At the cost of replacing $T$  by an \'etale neighborhood of $t$, we can suppose that $\sN$  extends to $T \setminus \{t\}$. We still denote by $\sN$ this extension.
Let $N$ be the rank of $\sN$.
Let us consider the sheaf $f^*\sN\otimes_{\Lambda}\sF$ on $U=X\setminus D$.
Since $f:X\longrightarrow T$ is smooth, it is transversal with respect to the extension by $0$ of $\sN$  to $T$.
From \cite[Prop. 2.22.2]{wr}, we deduce that $f^*\sN|_{\eta}$ is isoclinic with slope $s$.
Hence, by lemma \ref{tensordtsw} $(i)$, we have
\begin{equation*}
\dt_D(f^*\sN\otimes_{\Lambda}\sF)=N\cdot\rk_{\Lambda}\sF\cdot s.
\end{equation*}
From \cite[Prop. 4.1]{HY17}, we deduce
\begin{equation*}
\dt_x((f^*\sN\otimes_{\Lambda}\sF)|_{C_0}) \leq N\cdot\rk_{\Lambda}\sF\cdot s.
\end{equation*}
Notice that $f|_C:C\longrightarrow T$ is \'etale at $x$ since $C$ and $D$ intersect transversally at $x$.
Hence, the restriction $(f^*\sN)|_{C_0}$ is isoclinic with slope $s$ at $x$.
Lemma \ref{tensordtsw} $(i)$ then gives
\begin{equation*}
\dt_x((f^*\sN\otimes_{\Lambda}\sF)|_{C_0}) = N\cdot\rk_{\Lambda}\sF\cdot s.
\end{equation*}
Hence, lemma \ref{tensordtsw} $(i)$ ensures that $s>r(x,\sF|_{C_0})$.
Since this holds for every rational number $s>r(D,\sF)$ with $s=s_1/s_2$ and $(s_1,p)=(s_2,p)=1$, which is not a slope of $\sF|_{C_0}$ at $x$, we obtain that $r(D,\sF)\geq r(x,\sF|_{C_0})$.

(b)$\Rightarrow$(a). We only need to treat the case where $\dim_k X\geq 2$.  Let $x$ be a closed point in $D$ such that the ramification of $\sF$ is non-degenerate at $x$. In particular, the fiber of  $SS(j_{!}\sF)$ above $x$ is a finite number of lines $L_1, \dots, L_d\subset \bT^*_x X$. Since $\bT_xX$ is a $k$-vector space of dimension $\geq 2$, we can take a non-zero vector $L\in\bT_xX$ away from the hyperplane $\bT_xD\subset \bT_xX$ and dual hyperplanes $L_i^{\vee}\subset \bT_xX$ $(1\leq i\leq d)$.  Let $C$ be a smooth curve in $X$ passing through $x$ such that $\bT_xC=k\cdot L$. Then, the curve $C$ meets $D$ transversally and the immersion $i:C\longrightarrow X$ is $SS(j_!\sF)$-transversal at $x$. By proposition \ref{slope'} $(1)$, we obtain that $r(D,\sF)=r(x,\sF|_{C_0})$ for this curve $C$. It finishes the proof of (b)$\Rightarrow$(a).
\end{proof}

\begin{corollary}\label{slopehd}
Let $Y$ be a smooth connected $k$-scheme. Let $g:Y\longrightarrow X$ be a separated morphism such that $E:=D\times_XY$ is an irreducible smooth divisor in $Y$. We put $V=Y\setminus E$ and let $j^{\prime}: V\longrightarrow Y$ be the canonical injection. Then, we have
\begin{equation*}
r(E,\sF|_V)\leq r(D,\sF)
\end{equation*}
\end{corollary}
\begin{proof}
The morphism $g:Y\longrightarrow X$ is a composition of the graph embedding $\Gamma_g:Y\longrightarrow Y\times_k X$ and the projection $\pr_2:Y\times_k X\longrightarrow X$. Since $\pr_2$ is smooth, proposition 2.22.2 from \cite{wr}
implies
\begin{equation*}
r(D\times_kY, \sF|_{U\times_kY})=r(D,\sF).
\end{equation*}
We are thus reduced to treat the case where $g:Y\longrightarrow X$ is a closed immersion. By proposition \ref{slope'} $(2)$, we are left to show that, for any smooth $k$-curve in $Y$ meeting $E$ transversally at $y\in E$, we have $r(D,\sF)\geq r(y,\sF|_{C_0})$, where $C_0=C\backslash\{y\}$. Notice that $E=D\times_XY$ is smooth implies that $Y$ and $D$ meets transversally. Hence the smooth $k$-curve $C$ above meets $D$ transversally at $y$. By proposition \ref{slope'} $(2)$ again, we have $r(D,\sF)\geq r(y,\sF|_{C_0})$ indeed.
\end{proof}


\begin{proposition}\label{rlogDFrlogxFC}
The following two conditions are equivalent:
\begin{itemize}
\item[(a)]
$r\geq r_{\log}(D,\sF)$;
\item[(b)]
$(C,D)_x\cdot r\geq r_{\log}(x,\sF|_{C_0})$ for any smooth $k$-curve $C$ in $X$ meeting $D$ properly at $x$ such that the intersection number $(C,D)_x$ is co-prime to $p$, where $C_0=C\backslash\{x\}$.
\end{itemize}
\end{proposition}
\begin{proof}
(a)$\Rightarrow$(b). It is sufficient to show that, for a smooth $k$-curve $C$ in $X$ meeting $D$ properly at $x$ such that the intersection number $m=(C,D)_x$ is co-prime to $p$, we have
\begin{equation}\label{mrDFgeq rxFC}
m\cdot r_{\log}(D,\sF)\geq r_{\log}(x,\sF|_{C_0}). 
\end{equation}
At the cost of replacing $X$ by an \'etale neighborhood of $x$, we can suppose that there exists a smooth morphism $f:X\longrightarrow T$ with $T$ a smooth $k$-curve, and a closed point $t\in T$ such that $D=f^{-1}(t)$.
From lemma \ref{Nwithprescribedslope}, we can find a rational number $s$ such that $s>r_{\log}(D,\sF)$, that $m\cdot s$ is not a logarithmic slope of $\sF|_{C_0}$ at $x$ and that there exists a locally constant and constructible sheaf of $\Lambda$-modules $\sN$ on the generic point of $T_{(t)}$ such that $\sN$ is logarithmic isoclinic with logarithmic slope $s$.
At the cost of replacing $T$  by an \'etale neighborhood of $t$, we can suppose that $\sN$  extends to $T \setminus \{t\}$. We still denote by $\sN$ this extension.
Let us consider the sheaf $f^*\sN\otimes_{\Lambda}\sF$ on $U=X\setminus D$.
Since $f:X\longrightarrow T$ is smooth,
from \cite[Lemma 1.22]{logcc}, we deduce that $f^*\sN|_{\eta}$ is logarithmic isoclinic with isoclinic slope $s$.
Hence, by lemma \ref{tensordtsw} $(ii)$, we have
\begin{equation*}
\sw_D(f^*\sN\otimes_{\Lambda}\sF)=N\cdot\rk_{\Lambda}\sF\cdot s.
\end{equation*}
From \cite[Th. 6.6]{Hu}, we deduce
\begin{equation*}
\sw_x((f^*\sN\otimes_{\Lambda}\sF)|_{C_0}) \leq m\cdot N\cdot\rk_{\Lambda}\sF\cdot s.
\end{equation*}
Notice that $f|_C:C\longrightarrow T$ is tamely ramified at $x$ with the ramification index $m$, the restriction $(f^*\sN)|_{C_0}$ is isoclinic with slope $m\cdot s$ at $x$. Lemma \ref{tensordtsw} $(ii)$ then gives
\begin{equation*}
\sw_x((f^*\sN\otimes_{\Lambda}\sF)|_{C_0}) = m\cdot N\cdot\rk_{\Lambda}\sF\cdot s.
\end{equation*}
Hence, lemma \ref{tensordtsw} $(ii)$ ensures that $m\cdot s>r_{\log}(x,\sF|_{C_0})$.
Since this holds for every $s>r_{\log}(D,\sF)$ with $s=s_1/s_2$ and $(s_1,p)=(s_2,p)=1$ such that $m\cdot s$ is not a logarithmic slope of $\sF|_{C_0}$ at $x$, we obtain \eqref{mrDFgeq rxFC}.

(b)$\Rightarrow$(a). We only need to treat the case where $\dim_kX\geq 2$. Let $m$ be a positive integer co-prime to $p$.  By \cite[Prop. 6.4]{Hu}, after replacing $X$ by an affine Zariski neighborhood of $\xi$, we have two $k$-morphisms $i:C\longrightarrow X'$ and $g:X'\longrightarrow X$ satisfying:
\begin{itemize}
\item[(1)]
The scheme $X'$ is a smooth over $k$ with an integral and smooth divisor $D'=(g^*D)_{\mathrm{red}}$ and $g':X'\longrightarrow X$ is tamely ramified along $D$ with ramification index $m$;
\item[(2)]
The ramification of $\sF|_{U'}$ along $D'$ is non-degenerate, where $U'=g^{-1}(U)$;
\item[(3)]
The morphism $i:C\longrightarrow X'$ is an $SS(g^*j_!\sF)$-transversal immersion from a smooth $k$-curve to $X'$ such that $C$ and $D'$ meet transversally at a closed point $x\in D'$ and that the composition $g\circ i$ is also an immersion.  
\end{itemize}
Notice that $m=(C,D)_x$. By proposition \ref{slope'} (1), we have 
\begin{equation}\label{rx'FC>rD'FU'}
r_{\log}(x, \sF|_{C_0})+1=r(D', \sF|_{U'}),
\end{equation}
By proposition \ref{propramfil}, we have 
\begin{equation}\label{rD'FU'=mrlogDF}
r(D', \sF|_{U'})\geq r_{\log}(D', \sF|_{U'})=m\cdot r_{\log}(D, \sF).
\end{equation}
Combining \eqref{rx'FC>rD'FU'} and \eqref{rD'FU'=mrlogDF},  we get
\begin{equation}
\frac{r_{\log}(x, \sF|_{C_0})}{(C,D)_x}\geq  r_{\log}(D, \sF)-\frac{1}{m}.
\end{equation}
Hence the supremum of $\frac{r_{\log}(x, \sF|_{C_0})}{(C,D)_x}$ is greater than or equal to $r_{\log}(D, \sF)$, by taking all smooth $k$-curves $C$ in $X$ meeting $D$ properly at some closed point $x$ such that the intersection number $m=(C,D)_x$ is co-prime to $p$. We obtain (b)$\Rightarrow$(a).

\end{proof}


\begin{corollary}\label{logslopehd}
Let $Y$ be a smooth connected $k$-scheme. Let $g:Y\longrightarrow X$ be a separated morphism such that $E:=D\times_XY$ is an irreducible smooth divisor in $Y$. We put $V=Y\setminus E$ and let $j^{\prime}: V\longrightarrow Y$ be the canonical injection. Then, we have
\begin{equation*}
r_{\log}(E,\sF|_V)\leq r_{\log}(D,\sF)
\end{equation*}
\end{corollary}
\begin{proof}
It is enough to treat the case where $k$ is algebraically closed.
The morphism $g:Y\longrightarrow X$ is a composition of the graph embedding $\Gamma_g:Y\longrightarrow Y\times_k X$ and the projection $\pr_2:Y\times_k X\longrightarrow X$. Since $\pr_2$ is smooth, it is
$SS(j_!\sF)$-transversal. Hence, proposition 1.22 from \cite{{logcc}}
implies
\begin{equation*}
r_{\log}(D\times_kY, \sF|_{U\times_kY})=r_{\log}(D,\sF).
\end{equation*}
We are thus reduced to treat the case where $g:Y\longrightarrow X$ is a closed immersion. Applying proposition \ref{rlogDFrlogxFC} to the sheaf $\sF|_V$ ramified along $E$, we are sufficient to show that $(C,E)_x\cdot r_{\log}(D,\sF)\geq r_{\log}(x,\sF|_{C_0})$ for any smooth $k$-curve $C$ in $Y$ meeting $E$ properly at $x$ such that the intersection number $(C,Y)_x$ is co-prime to $p$, where $C_0=C\backslash\{x\}$. Notice that $Y\subseteq X$ is a closed subscheme and that $Y$ and $D$ meet transversally. Hence the smooth $k$-curve $C$ above is a curve in $X$ meeting $D$ properly at $x$ with $(C,D)_x=(C,E)_x$. Applying proposition \ref{slope'} $(2)$ to the sheaf $\sF$ ramified along $D$, we have $(C,D)_x\cdot r_{\log}(D,\sF)\geq r_{\log}(x,\sF|_{C_0})$. We finish the proof of this corollary.
\end{proof}

\section{Characteristic cycle of a sheaf after twist}

\begin{lemma}\label{makerelcurve}
 Let $k$ be an algebraically closed field of characteristic $p>0$. Let $X$ be a smooth connected $k$-scheme of dimension $d\geq 2$. Let $D$ be a smooth connected divisor in $X$. Let $x$ be a closed point in $D$. Let $g:X\longrightarrow \bA^1_k$ be a smooth morphism such that $g|_D:D\longrightarrow\bA^1_k$ is smooth. Then, at the cost of replacing $X$ by a suitable Zariski neighborhood of $x$, there exists a map $h:X\longrightarrow \bA^{d-1}_k$ such that
\begin{itemize}
\item[(1)]
$g$ is the composition of $h$ and the projection $\pr_1:\bA^{d-1}_k\longrightarrow\bA^1_k$;
\item[(2)]
$h:X\longrightarrow \bA^{d-1}_k$ is  smooth and the restriction $h|_D:D\longrightarrow \bA^{d-1}_k$ is \'etale.
\end{itemize}
\end{lemma}
\begin{proof}
Suppose that the image of $x$ in $\bA^1_k$ is the origin $o\in\bA^1_k$. Consider the canonical maps
\begin{equation}\label{mapcotan}
\bT^*_o\bA^1_k \longrightarrow \bT^*_xX \longrightarrow \bT^*_xD
\end{equation}
We denote by $t_1,\ldots,t_d$ a basis of $\bT^*_xX$, by $s_1,\ldots, s_{d-1}$ a basis of $\bT^*_xD$ such that $t_i$ maps to $s_i$ for $1\leq i\leq d-1$ and such that $t_d$ maps to $0$. Let $l$ be a basis of the line $\bT^*_o\bA^1_k$. Let  $\sum_{1\leq i\leq d}\lambda_it_i$ ($\lambda_i\in k$) be the image of $l$ by $\bT^*_o\bA^1_k\longrightarrow \bT^*_xX$. Since $g|_D:D\longrightarrow \bA^1_k$ is smooth, the composition \eqref{mapcotan} is injective. Hence, $(\lambda_1,\ldots\lambda_{d-1})\neq 0$. We may assume that $\lambda_{1}\neq 0$. Let $\wt t_2,\ldots\wt t_{d-1}$ be regular functions on $X$ defined around $x$ and lifting $t_2,\ldots,t_{d-1}$ respectively. We have the following morphism of $k$-algebras
\begin{equation}\label{morphi}
k[y,x_2,\ldots,x_{d-1}]\longrightarrow\sO_{X,x}\ \ \   y\mapsto g^*(y), \ \ x_i\mapsto \wt t_i.
\end{equation}
In a neighbourhood of $x$, the map \eqref{morphi} induces a map
\begin{equation*}
h:X\longrightarrow \bA^{d-1}_k.
\end{equation*}
 It induces an injection $d h_x:\bT^*_o\bA^{d-1}\longrightarrow \bT^*_xX$ and an isomorphism $(d h|_D)_x:\bT^*_o\bA^{d-1}_k\longrightarrow\bT^*_xD$. Hence, $h:X\longrightarrow \bA^{d-1}_k$ is smooth at $x$ and $h|_D:\bA^{d-1}_k\longrightarrow D$ is \'etale at $x$. By construction, the composition of $h$ and $\pr_1$ is $g$.
\end{proof}

\begin{proposition}\label{constant_swan_implies_easy_SS}
Let $k$ be a perfect field of characteristic $p>0$.
Let $X$ be a smooth connected $k$-scheme.
Let $D$ be a smooth connected divisor in $X$.
Put $U:=X\setminus D$.
Let $j : U \longrightarrow X$ be the canonical injection.
Let $\Lambda$ be a finite field of characteristic $\ell\neq p$.
Let $\sF$ be a non zero locally constant and constructible sheaf of $\Lambda$-modules on $U$.
Suppose that for every smooth curve $Z\subset X$ meeting $D$ transversally at a point $x$, the Swan conductor $\sw_x(\sF|_Z)$ is constant depending neither on $x$ nor on $Z$.
Then,
$$
SS j_!\sF=\bT^*_{X}X\bigcup \bT^*_DX
$$
\end{proposition}
\begin{proof}
We may assume that $k$ is algebraically closed.
We first show
\begin{equation}\label{inclusion_SS}
SS j_!\sF\subset\bT^*_{X}X\bigcup \bT^*_DX.
\end{equation}
Let $x$ be a closed point of $D$.
Let $g:X\longrightarrow \bA^1_k$ be a $k$-morphism which is $(\bT^*_{X}X\bigcup \bT^*_DX)$-transversal at $x$.
To prove \eqref{inclusion_SS}, we need to show by Beilinson's Theorem \ref{beilinsonmain} that $g:X\longrightarrow \bA^1_k$ is universally locally acyclic with respect to $j_!\sF$ in an open neighborhood of $x$.
By assumption, $g:X\longrightarrow \bA^1_k$ and $g|_D:D\longrightarrow \bA^1_k$ are smooth at $x$.
At the cost of replacing ${X}$ by a Zariski neighborhood of $x$, lemma \ref{makerelcurve} ensures that  $g=\pr_1 \circ h$ where $\pr_1:\bA^{d-1}_k\longrightarrow \bA^1_k$  is the first projection and $h:X\longrightarrow\bA^{d-1}_k$ is a smooth morphism such that $h(x)=0$ and $h|_D:D\longrightarrow \bA^{d-1}_k$ is \'etale.
For any closed point $z\in \bA^{d-1}_k$ in a sufficiently small neighbourhood of $0$, $T_z=h^{-1}(z)$ is a smooth $k$-curve transverse to $D$.
By assumption, for every closed point $w\in T_z\bigcap D$, the Swan conductor $\sw_w(\sF|_{T_z})$ does not depend on $z$ nor $w$.
Applying Deligne and Laumon's semi-continuity theorem for Swan conductors \cite[Th. 2.1.1]{lau} to the sheaf $j_!\sF$ on the relative curve $h:X\longrightarrow\bA^{d-1}_k$, we obtain that $h:X\longrightarrow\bA^{d-1}_k$ is universally locally acyclic with respect to $j_!\sF$.
Since $\pr_1:\bA^{d-1}_k\longrightarrow \bA^1_k$ is smooth, lemma 7.7.6 from \cite{fu} ensures that $g=\pr_1 \circ h$ is universally locally acyclic with respect to $j_!\sF$.
Since this holds in the neighbourhood of every closed point of $D$, we obtain that the sheaf $j_!\sF$ is weakly micro-supported on $\bT^*_{X}X\bigcup \bT^*_DX$.
From Beilinson's Theorem \ref{beilinsonmain}, we deduce that the inclusion \eqref{inclusion_SS} holds.
On the other hand, since $\sF\neq 0$, the sheaf $j_!\sF$ is not constant.
Hence, the singular support $SS j_!\sF$ must contain irreducible components of dimension $\dim_k X$ supported on $D$.
Proposition \ref{constant_swan_implies_easy_SS} thus follows.
\end{proof}

\begin{theorem}\label{sameCC}
Let $k$ be a perfect field of characteristic $p>0$. Let $X$ be a smooth connected $k$-scheme. Put $d=\dim_k X$. Let $D$ be a smooth connected divisor in $X$. Put $U:=X\setminus D$. Let $j : U \longrightarrow X$ be the canonical injection.
Let $S$ be a connected smooth $k$-curve. Let $s$ be a closed point in $S$. Put  $V:=S\setminus \{s\}$. Let  $f:X\longrightarrow S$ be a smooth morphism such that $D=f^{-1}(s)$. Let $\Lambda$ be a finite field of characteristic $\ell\neq p$. Let $\sF$ be a non-zero locally constant and constructible sheaf of $\Lambda$-modules on $U$. Let $\sN$  be a locally constant and constructible sheaf of $\Lambda$-modules on $V$ of rank $N>0$. Suppose that $\sN$ has pure logarithmic slope $r_{\log}$ at $s$ with $r_{\log}>r(D,\sF)-1$. Then, we have
\begin{align*}
SS(j_!(\sF\otimes_{\Lambda} f^*\sN))&= SS(j_! f^*\sN)=\bT^*_{X}X\bigcup \bT^*_DX\\
CC(j_!(\sF\otimes_{\Lambda} f^*\sN))&=\rk_{\Lambda}\sF\cdot CC(j_!f^*\sN)\\
&=(-1)^d\big(N\cdot\rk_{\Lambda} \sF\cdot[\bT^*_XX]+(r_{\log}+1)N\cdot\rk_{\Lambda} \sF\cdot[\bT^*_DX]\big)
\end{align*}
\end{theorem}
\begin{proof}
We may assume that $k$ is algebraically closed and that $\dim_k X\geq 2$.
The equality
$$
SS(j_! f^*\sN)=\bT^*_{X}X\bigcup \bT^*_DX
$$
follows from \cite[Th. 1.4]{bei}.
Let $Z\subset X$ be a smooth curve meeting $D$ transversally at a point $x$.
Then, $f|_{Z}:Z\longrightarrow S$ is \'etale in a neighborhood of $x$.
Thus, $(f^*\sN)|Z$ has pure logarithmic slope $r_{\log}$ at $x$.
From lemma \ref{tensordtsw} and proposition \ref{slope'}, we deduce
\begin{equation}\label{dimtotw}
\dt_x(j_!(\sF\otimes_{\Lambda} f^*\sN)|_{Z})=(r_{\log}+1)N\cdot\rk_{\Lambda} \sF.
\end{equation}
In particular, the left-hand side of \eqref{dimtotw} does not depend on $x$ nor on $Z$.
Then, the first part of Theorem \ref{sameCC} follows from proposition \ref{constant_swan_implies_easy_SS}.\\ \indent
We now compute the characteristic cycles $CC(j_!(\sF\otimes_{\Lambda} f^*\sN))$ and $CC(j_!f^*\sN)$. Applying Theorem \ref{cc7.6} to the smooth
morphism $f:X\longrightarrow S$, we have
\begin{equation}
CC(j_!f^*\sN)=(-1)^d\big(N\cdot[\bT^*_XX]+(r_{\log}+1)N\cdot[\bT^*_DX]\big).
\end{equation}
From \eqref{CCnondeg}, the coefficient of $\bT^*_{X}{X}$ in $CC(j_!(\sF\otimes_{\Lambda} f^*\sN))$ is $(-1)^d N\cdot\rk_{\Lambda} \sF$. We are thus left to compute the coefficient of $\bT^*_{D}{X}$ in $CC(j_!(\sF\otimes_{\Lambda} f^*\sN))$. Let $i : C\longrightarrow X$ be a smooth curve in $X$ transverse to $D$.
Since
$$
SS(j_!(\sF\otimes_{\Lambda} f^*\sN))=\bT^*_{X}X\bigcup \bT^*_DX,
$$
the curve  $i : C\longrightarrow X$ is properly $SS(j_!(\sF\otimes_{\Lambda} f^*\sN))$-transversal.
By Theorem \ref{cc7.6} again, we have
\begin{equation}\label{CCcutcurve}
i^!CC(j_!(\sF\otimes_{\Lambda} f^*\sN))=CC(i^*j_!(\sF\otimes_{\Lambda} f^*\sN)).
\end{equation}
Hence, example \ref{CCcurve} and proposition \ref{slope'} imply that the sought-after coefficient is
\begin{equation*}
(-1)^d(r_{\log}+1)N\cdot\rk_{\Lambda} \sF,
\end{equation*}
which finishes the proof of Theorem \ref{sameCC}.
\end{proof}

\begin{lemma}\label{!=*criterion}
Let $k$ be a perfect field of characteristic $p>0$.
Let $X$ be a smooth $k$-scheme.
Let $D$ be a smooth divisor in $X$.
Put $U:=X\setminus D$.
Let $j : U \longrightarrow X$ be the canonical injection.
Let $\sF$ be a locally constant and constructible sheaf of $\Lambda$-modules on $U$ such that
$SS j_! \sF \subset \bT^*_{X}X\bigcup \bT^*_DX$.
Suppose that for every smooth curve $C\subset X$ meeting $D$ transversally at a point $x$, we have $(Rj_{C\ast}\sF|_{C\cap U})_x=0$ where $j_C : C\cap U \longrightarrow C$ denotes the open immersion.
Then, $j_! \sF \longrightarrow Rj_*\sF$  is an isomorphism.
\end{lemma}
\begin{proof}
Let $x$ be a closed point of $D$.
We have to show that the germ of $Rj_*\sF$ at $x$ vanishes.
Let $i:C\longrightarrow X$ be a locally closed immersion from a smooth $k$-curve $C$ to $X$ such that $C$ meets $D$ transversally at $x$.
Consider the following cartesian diagram
\begin{equation*}
\xymatrix{\relax
C\cap U\ar[r]^-{i_C}\ar[d]_{j_C}\ar@{}|{\Box}[rd]&U\ar[d]^j\\
C\ar[r]_i&{X}}
\end{equation*}
Since $SS j_! \sF \subset \bT^*_{X}X\bigcup \bT^*_DX$, the morphism $i:C\longrightarrow X$ is $SS(j_! \sF)$-transversal.
Since $j_! \sF$ is perverse up to a shift, proposition \ref{51314} implies that $SS(j_! \sF)=SS(Rj_*\sF)$.
Hence, the morphism $i:C\longrightarrow X$ is $SS(Rj_*\sF)$-transversal.
From corollary \ref{basechange}, we deduce
\begin{equation}\label{vanish1}
i^*Rj_*\sF\cong Rj_{C*}\sF|_{C\cap U}
\end{equation}
Lemma \ref{!=*criterion} thus follows by assumption on $\sF$.
\end{proof}

\begin{corollary}\label{!=R*}
Under the conditions of Theorem \ref{sameCC}, we have
\begin{equation*}
Rj_*(\sF\otimes_{\Lambda} f^*\sN)=j_!(\sF\otimes_{\Lambda} f^*\sN).
\end{equation*}
\end{corollary}
\begin{proof}
We may assume that $k$ is algebraically closed.
Let $i:C\longrightarrow X$ be a locally closed immersion from a smooth $k$-curve $C$ to $X$ such that $C$ meets $D$ transversally at $x$.
Let $j_C : C\cap U \longrightarrow C$ be the open immersion.
Notice that $(f^*\sN)|_{C\cap U}$ has pure slope $r_{\log}+1$ at $x$ since $f|_C:C\longrightarrow S$ is \'etale in a neighborhood of $x$.
From proposition \ref{slope'}, the largest slope of $\sF|_{C\cap U}$ at $x$ is no greater than $r(D,\sF)$.
Hence $(\sF\otimes_{\Lambda} f^*\sN)|_{C\cap U}$ is totally wild ramified at $x$ of pure slope $r_{\log}+1$.
By \cite[Prop. 8.1.4]{fu}, we deduce
\begin{equation}\label{vanish2}
Rj_{C*}((\sF\otimes_{\Lambda} f^*\sN)|_{C\cap U})= j_{C!}((\sF\otimes_{\Lambda} f^*\sN)|_{C\cap U}).
\end{equation}
Hence, corollary \ref{!=R*} follows from the combination of Theorem \ref{sameCC} with lemma \ref{!=*criterion}.
\end{proof}

\section{Bound for the wild ramification of nearby cycles}

\subsection{}\label{defnc}
In this section, let $R$ be a henselian discrete valuation ring. Let $K$ be the fraction field of $R$ and let $F$ be the residue field of $R$. Suppose that $F$ is perfect of characteristic $p>0$. Put $S=\spec(R)$. Let $s$ be the closed point of $S$. Let $\bar s\longrightarrow S$ be a geometric point above $s$. Let $\eta$ be the generic point of $S$, $\eta^{\rt}$ a maximal tame cover of $\eta$, $\bar\eta$ a geometric point above $\eta^{\rt}$. Let $G$ be the Galois group of $\bar\eta$ over $\eta$. Let $P$ be the wild inertia subgroup of $G$.

Let $f:X\longrightarrow S$ be a morphism of schemes, $f_{\eta}:X_{\eta}\longrightarrow \eta$ its generic fiber and $f_s:X_s\longrightarrow s$ its closed fiber and $f_{\bar s}:X_{\bar s}\longrightarrow \bar s$ its geometric closed fiber.  Let $\Lambda$ be a finite field of characteristic $\ell\neq p$. For an object $\sF$ in $D^{+}(X_{\eta},\Lambda)$, we denote by $R\Psi(\sF,f)$ (resp. $R\Psi^{\rt}(\sF,f)$) the nearby cycles complex (resp. the tame nearby cycles complex) of $\sF$  with respect to $f:X\longrightarrow S$. For an object $\sF$ in $D^{+}(X,\Lambda)$, we denote by $R\Phi(\sF,f)$ the vanishing cycles complex of $\sF$  with respect to $f:X\longrightarrow S$. All these complex are objects of $D^+(X_{\bar s}, \Lambda)$ with $G$-actions.

Note that $R\Psi$, $R\Psi^{\rt}$ and $R\Phi$ also make sense when $S$ is a smooth curve over a perfect field with a closed point $s$. For nearby cycles and vanishing cycles on a more general base scheme, we refer to \cite{Orgogozo}.


The following definition was introduced in  \cite{tey}.

\begin{definition}
Let $\sF$ be an object in $D^b_c(X_{\eta},\Lambda)$. Let $\bar x\longrightarrow X_{\bar s}$ be a geometric point above a closed point $x\in X_{\bar s}$. We say that $r\in \mathbb{Q}_{\geq 0}$ is a {\it nearby slope for $\sF$ with respect to $f:X\longrightarrow S$ at $x$} if there is a locally constant and constructible sheaf $\sN$ of $\Lambda$-modules on $\eta$ with pure logarithmic slope $r$ such that
\begin{equation*}
R\Psi^{\rt}_{\bar x}(\sF\otimes_{\Lambda} f_{\eta}^*\sN,f)\neq 0.
\end{equation*}
We denote by $\Sl(\sF,f,x)$ the set of nearby slopes of $\sF$ with respect to $f:X\longrightarrow S$ at $x$.
\end{definition}

\begin{lemma}\label{tnc}
Suppose that $f:X\longrightarrow S$ is of finite type. Let  $\sF$ be an object in $D^b_c(X_{\eta},\Lambda)$. Let $\sN$ be an object in $D^b_c(\eta,\Lambda)$. Then, for any geometric point $\bar x\longrightarrow X_{\bar s}$ above a closed point $x$ in $X_{\bar s}$, there is a canonical $G/P$-equivariant isomorphism
\begin{equation}\label{tnc=ncp}
R\Psi^{\rt}_{\bar x}(\sF\otimes^L_{\Lambda}f_{\eta}^*\sN,f)\cong (R\Psi_{\bar x}(\sF,f)\otimes^L_{\Lambda}\sN|_{\bar\eta})^P.
\end{equation}
\end{lemma}
\begin{proof}
Consider the following cartesian diagram
\begin{equation*}
\xymatrix{\relax
X_{(\bar x)}\times_S\bar\eta\ar[r]\ar@{}|(0.5){\Box}[rd]\ar[d]_{f_{\bar \eta}}&X_{(\bar x)}\times_S \eta^{\rt}\ar[d]^{f_{\eta^{\rt}}}\ar[r]\ar@{}|(0.5){\Box}[rd]&X_{(\bar x)}\times_S \eta\ar[d]^{f_{\eta}}\ar[r]\ar@{}|(0.5){\Box}[rd]&X_{(\bar x)}\ar[d]^{f}\ar@{}|(0.5){\Box}[rd]&X_{(\bar x)}\times_S \bar s\ar[l]\ar[d]^{f_{ \bar s}}\\
\bar\eta\ar[r]&\eta^{\rt}\ar[r]&\eta\ar[r]&S&{\bar s}\ar[l]
}
\end{equation*}
Then we have
\begin{align}
R\Psi^{\rt}_{x}(\sF\otimes^L_{\Lambda}f_{\eta}^*\sN)&\cong R\Gamma(X_{(\bar x)}\times_S\eta^{\rt}, (\sF\otimes^L_{\Lambda}f_{\eta}^*\sN)|_{X_{(\bar x)}\times\eta^{\rt}})\\
&=R\Gamma(\eta^{\rt}, Rf_{\eta^{\rt} *}(\sF|_{X_{(\bar  x)}\times_S\eta^{\rt}}\otimes^L_{\Lambda}f^*_{\eta^{\rt}}(\sN|_{\eta^{\rt}})))\nonumber\\
&\cong R\Gamma(\eta^{\rt}, Rf_{\eta^{\rt} *}(\sF|_{X_{(\bar x)}\times_S\eta^{\rt}})\otimes^L_{\Lambda}\sN|_{\eta^{\rt}})\nonumber\\
&=R\Gamma(P, Rf_{\bar\eta *}(\sF|_{X_{(\bar x)}\times_S\bar\eta})\otimes^L_{\Lambda}\sN|_{\bar\eta})\nonumber\\
&=(Rf_{\bar\eta *}(\sF|_{X_{(\bar x)}\times_S\bar\eta})\otimes^L_{\Lambda}\sN|_{\bar\eta})^P\nonumber\\
&=(R\Gamma(X_{(\bar x)}\times_S\bar\eta, \sF|_{X_{(\bar x)}\times_S\bar\eta})\otimes^L_{\Lambda}\sN|_{\bar\eta})^P\nonumber\\
&=(R\Psi_{\bar x}(\sF,f)\otimes^L_{\Lambda}\sN|_{\bar\eta})^P,\nonumber
\end{align}
where the third identification comes from the projection formula
\begin{equation*}
Rf_{\eta^{\rt} *}(\sF|_{X_{(x)}\times_S\eta^{\rt}}\otimes^L_{\Lambda}f^*_{\eta^{\rt}}(\sN|_{\eta^{\rt}}))\cong Rf_{\eta^{\rt} *}(\sF|_{X_{(x)}\times_S\eta^{\rt}})\otimes^L_{\Lambda}\sN|_{\eta^{\rt}},
\end{equation*}
which is valid since $Rf_{\eta^{\rt}*}$ has finite cohomological dimension and each cohomology sheaf of $\sN$ is locally constant.
\end{proof}

\begin{corollary}\label{tncls}
Suppose that $f:X\longrightarrow S$ is of finite type. Let $\sF$ be an object in $D^b_c(X_{\eta},\Lambda)$. Let $\bar x\longrightarrow X_s$ be a geometric point above a closed point $x$ in $X_{\bar s}$.
Then
\begin{equation}\label{equality_slopes}
\Sl(\sF,f,x)=\bigcup_{i\in \mathbb Z} \Sl_{\log, K}(R^i\Psi_{\bar x}(\sF,f))
\end{equation}
\end{corollary}
\begin{proof}
Put
$$
M=\bigoplus_{i\in \mathds{Z}} R^i\Psi_{\bar x}(\sF,f).
$$
If $M=0$, then lemma \ref{tnc} ensures that both sides of  \eqref{equality_slopes} are empty so that lemma \ref{tncls} holds in that case.
We can thus suppose $M\neq 0$.
Let $s\geq 0$.
Let $N$ be a finitely generated $\Lambda$-module with continuous $G$-action.
Assume that $N$ has pure logarithmic slope $s$.
We denote by $\sN$  the locally constant and constructible sheaf of $\Lambda$-modules on $\eta$ associated to $N$.
By lemma \ref{tnc}, we have
\begin{equation}\label{big_equality}
\bigoplus_{i\in \mathds{Z}}  R^i\Psi^{\rt}_{\bar x}(\sF\otimes^L_{\Lambda}f_{\eta}^*\sN,f)\cong (M\otimes_{\Lambda}N )^P.
\end{equation}
If $s$ is not a nearby slope of $\sF$ with respect to $f$ at $x$, then $(M\otimes_{\Lambda}N )^P=0$ for every $N$ as above.
Then, lemma \ref{added_referee_lemma} $(2)$ implies that $s$ is not a logarithmic slope for $M$.
Conversely, suppose that $s$ is not a logarithmic slope of $M$.
Then, lemma \ref{added_referee_lemma} $(2)$ implies that the left hand side of \eqref{big_equality} vanishes for every locally constant and constructible sheaf of $\Lambda$-modules $\sN$ on $\eta$ with pure logarithmic slope $s$.
Hence, $s$ is not a nearby slope of $\sF$ with respect to $f$ at $x$.
This finishes the proof of corollary \ref{tncls}.
\end{proof}

\begin{lemma}\label{modularepresentation}
Suppose that $f:X\longrightarrow S$ is of finite type. Let $\sF$ be a constructible sheaf of $\Lambda$-module on $X_{\eta}$. We assume that there is a sequence of subsheaves 
$$0=\sF_0\subset\sF_1\subset \cdots\subset \sF_n=\sF$$
such that $\sF_1\cong \sF_r/\sF_{r-1}$ for any $1\leq r\leq n$. Let $H$ be a closed subgroup of $P$. Let $\bar x\longrightarrow X_{\bar s}$ be a geometric point. The action of $H$ on each $R^i\Psi_{\bar x}(\sF,f)$ is trivial if and only if that of $H$ on each $R^i\Psi_{\bar x}(\sF_1,f)$ is also trivial ($i\in\mathbb Z$).
\end{lemma}

\begin{proof}
Since $H$ is a pro-$p$ group and $\ch(\Lambda)=\ell$  $(\ell\neq p)$, every $\Lambda$-modules with continuous $H$-actions are semi-simple. 

We first prove the if part. Suppose an integer $1\leq r\leq n-1$ satisfies that the action of $H$ on each $R^i_{\bar x}\Psi(\sF_r,f)$ is trivial. The short exact sequence 
$$0\longrightarrow \sF_r\longrightarrow\sF_{r+1}\longrightarrow\sF_1\longrightarrow 0$$
gives a long exact sequence 
$$\cdots \longrightarrow R^i\Psi_{\bar x}(\sF_r,f)\longrightarrow R^i\Psi_{\bar x}(\sF_{r+1},f)\longrightarrow R^i\Psi_{\bar x}(\sF_1,f)\longrightarrow R^{i+1}\Psi_{\bar x}(\sF_r,f)\longrightarrow\cdots$$
Since the actions of $H$ on each $R^i\Psi_{\bar x}(\sF_r,f)$ and each $R^i\Psi_{\bar x}(\sF_1,f)$ are trivial, that of $H$ on each $R^i\Psi_{\bar x}(\sF_{r+1},f)$ are also trivial, for each $i\in\mathbb Z$. By the induction on $r$, we finish the proof of the if part.

Now we treat the only if part. Let $A$ be the set of integers such that, for any $i\in A$ and any $1\leq r\leq n$, the action of $H$ on $R^i\Psi_{\bar x}(\sF_r,f)$ is trivial. It is easy to see that $\bZ_{\leq -1}\subseteq A$. We take $q\in A$, for any $1\leq r\leq n$, we have an exact sequence 
$$\cdots\longrightarrow R^q\Psi_{\bar x}(\sF_1,f)\longrightarrow R^{q+1}\Psi_{\bar x}(\sF_r,f)\longrightarrow R^{q+1}\Psi_{\bar x}(\sF_{r+1},f)\longrightarrow \cdots$$
Since actions of $H$ on $R^{q}\Psi_{\bar x}(\sF_1,f)$ and on $R^{q+1}\Psi_{\bar x}(\sF_n,f)$ is trivial, then, that of $H$ on each $R^{q+1}\Psi_{\bar x}(\sF_r,f)$ $(1\leq r\leq n)$ is trivial. Hence $q+1\in A$. By the induction on $q$, we see that $A=\mathbb Z$. We finish the proof of the only if part. 
\end{proof}

\subsection{}\label{defsemistable}
Let $Z$ be a reduced subscheme of $X$ containing $X_s$. We say that $(X,Z)$ is a {\it semi-stable pair over} $S$ if $f:X\longrightarrow S$ is of finite type and if, \'etale locally, $X$ is \'etale over $\spec(R[t_1,\ldots, t_d]/(t_{r+1}\cdots t_d-\pi))$, where  $r<d$ and $Z=X_s\cup Z_f$ with $Z_f$ defined by $t_1\cdots t_m=0$,  where $m\leq r$.

\vspace{0.2cm}
 We now prove the main theorem of this article.
 
\begin{theorem}\label{genleal}
Let $k$ be a perfect field of characteristic $p>0$. Suppose that  $S$ is the henselization at a closed point of a smooth curve over $k$. Let $(X,Z)$ be a semi-stable pair over $S$ such that $f : X\longrightarrow S$ is smooth. Let $\sF$ be a locally constant and
constructible sheaf of $\Lambda$-modules on $U:=X\setminus Z$. Let $j : U\longrightarrow X$ be the canonical injection. Suppose that $\sF$ is tamely ramified  along the horizontal part of $Z$. Let $r_{\log}(\sF)$ be the maximum of the set of logarithmic slope of $\sF$ at generic points of the special fiber $X_s$ of $X$.

Then, for every $r>r_{\log}(\sF)$, the $r$-th upper numbering ramification subgroup of $G$ acts trivially on $R^i\Psi(j_{!}\sF,f)$ for every $i\in \mathbb Z_{\geq 0}$.
\end{theorem}
\begin{proof}

 We may assume that $k$ is algebraically closed and that the special fiber $X_s$ is irreducible. To show that the action of $G^{r_{\log}(\sF)+}_{K,\cl}$ on each $R^i\Psi(\sF,f)$ is trivial is equivalent to show that for every $i\in \mathbb{Z}$, for every closed point $x$ of $X_s$, we have
\begin{equation}\label{rnc<r}
r_{\log, K}(R^i\Psi_{x}(\sF,f))\leq r_{\log}(\sF).
\end{equation}
Step 1. We first prove \eqref{rnc<r} in the case where $U=X_\eta$, that is $Z_f=\emptyset$. From corollary \ref{tncls}, we have to prove that for any locally constant and constructible sheaf $\sN$ of $\Lambda$-modules on $\eta$ with pure logarithmic slope $r_{\log}>r_{\log}(\sF)$, we have
\begin{equation}\label{psitameiszero}
R\Psi^{\rt}(\sF\otimes_{\Lambda}f^*_{\eta}\sN,f)=0.
\end{equation}
Let $n>0$ be an integer prime to $p$. Let $\varpi$ be a uniformizer of $R$. Put
\begin{equation*}
S_n=\spec(R[t]/(t^n-\varpi)),
\end{equation*}
 put ${X}_n={X}\times_SS_n$ and put $U_n=U\times_SS_n$. Let $\eta_n$ be the generic point of $S_n$. We thus have the following commutative diagram
\begin{equation*}
\xymatrix{\relax
U_n\ar[r]^{j_n}\ar[d]_{h_n}\ar@{}|-{\Box}[rd]&{X}_n\ar[d]^{g_n}&X_s\ar@{=}[d]\ar[l]_{i_n}\\
U\ar[r]_{j}&{X}&X_s\ar[l]^i}
\end{equation*}
By \cite[VII 5.11]{SGA4-2}, we have
\begin{equation}\label{psimod}
R\Psi^{\rt}(\sF\otimes_\Lambda f_{\eta}^*\sN)=\varinjlim_{(n,p)=1}i^*_nRj_{n*}(h^*_n\sF\otimes_\Lambda h^*_nf_{\eta}^*\sN).
\end{equation}
To show the vanishing \eqref{psitameiszero}, it is thus enough to show that for any integer $n>0$ prime to $p$, we have
\begin{equation}\label{tnc*=!}
Rj_{n*}(h^*_n\sF\otimes_\Lambda h^*_nf_{\eta}^*\sN)=j_{n!}(h^*_n\sF\otimes_\Lambda h^*_n f_{\eta}^*\sN).
\end{equation}
Since $S$ is the henselization of a smooth curve over $k$, since $f$ is of finite type and since $\sF$ and $\sN$ are constructible, we may descend to situation where $S$ is a smooth $k$-curve. From corollary \ref{!=R*}, to show \eqref{tnc*=!}, it is enough to show that
\begin{equation*}
r(X_s,h^*_n\sF)-1<r_{\log}(s,\sN|_{\eta_n})
\end{equation*}
From (ii) and (iii) of proposition \ref{propramfil}, we have
\begin{align}
r(X_s, h_n^*\sF)-1 \leq r_{\log}(X_s, h_n^*\sF) &= n\cdot  r_{\log}(X_s,\sF)=n\cdot r_{\log}(\sF)
\end{align}
Furthermore, $\sN|_{\eta_n}$ has pure logarithmic slope $n\cdot r_{\log}$ at $s$.  Hence,
\begin{equation*}
r(X_s, h_n^*\sF)-1<    n\cdot  r_{\log}=r_{\log}(s,\sN|_{\eta_n})
\end{equation*}
This concludes the proof of Theorem \ref{genleal} in the case where $Z$ has no horizontal part.

Step 2. We now prove the general case. We need to prove \eqref{rnc<r} for a fixed closed point $x$ of $X_s$. As Step 1, we may descend to the case where $S$ is an affine smooth $k$-curve with local coordinate $\varpi$ at the closed point $s$. Since this is an \'etale local question, we may assume
that $X$ is affine and \'etale over $\spec(\sO_S[t_1,\ldots, t_d]/(t_d-\varpi))$ such that the image of $x$ is the point associated to the ideal $(t_1,\ldots, t_d)$, and $Z=X_s\cup Z_f$ with $Z_f$ defined by $t_1\cdots t_m=0$,  where $m < d$. We denote by $\{Z_i\}_{i\in I}$ the set of irreducible components of $Z$ passing through $x$ and, for any subset $J\subseteq I$, we put $Z_J=\cap_{i\in J}Z_i$. Since each $Z_J$ is smooth at $x$, at the cost of replacing $X$ by a Zariski neighborhood of $x$, we may assume that the $Z_J$'s are connected and smooth.

 From Abhyankar's lemma \cite[XIII 5.2]{SGA1}, there exist prime to $p$ integers $n_1,\ldots n_m$ such that after setting
\begin{equation*}
X'=\spec(\sO_X[t'_1,\ldots, t'_m]/(t'^{n_1}_1-t_1,\ldots, t'^{n_m}_m-t_m))
\end{equation*}
the locally constant sheaf $\sF|_{X'\times_XU}$ is unramified at the generic points of $X'\times_X Z_f$. Consider the following cartesian diagram
\begin{equation*}
\xymatrix{\relax
U^{\prime}   \ar[r]^-{j^{\prime}}\ar@{}|(0.5){\Box}[rd]   \ar[d]_-{\pi_U} & X^{\prime}\setminus X'_s\ar[d] \ar[r]\ar@{}|(0.5){\Box}[rd]   &    X^{\prime} \ar[d]^-{\pi}  \\
                          U  \ar[r]^-{j}   &  X \setminus X_s  \ar[r]    &  X
                          }
\end{equation*}
where $X'_s=X'\times_Ss$.
By Zariski-Nagata's purity theorem, $\sF|_{X'\times_XU}=\pi_U^{\ast}\sF$ extends into a locally constant and constructible sheaf of $\Lambda$-modules $\sF'$ on $X'\setminus X'_s$.  Note that $\pi_U : U'\longrightarrow U$ is finite Galois \'etale. Hence, 
$$
\pi_{U\ast}\pi^{\ast}_U\sF\cong \sF\otimes_{\Lambda}\pi_{U\ast}\pi^{\ast}_U\Lambda.
$$ 
Since the cardinality of $\gal(U'/U)$ may be divisible by $\ell$, the locally constant sheaf $\pi_{U\ast}\pi^{\ast}_U\Lambda$ has a direct summand $\mathscr H$ which is a finite successive extension of the constant sheaf $\Lambda$. Hence $\pi_{U\ast}\pi^{\ast}_U\sF$ has a direct summand $\sF\otimes_{\Lambda}\mathscr H$ which is a successive extension of $\sF$. By lemma \ref{modularepresentation}, to show 
$$
r_{\log, K}(R^i\Psi_{{x}}(j_!\sF,f))\leq r_{\log}(X_s,\sF)
$$ 
for all $i\in\mathbb Z$, it is sufficient to show 
$$
r_{\log, K}(R^i\Psi_{{x}}(j_!\pi_{U*}\pi^{\ast}_U\sF,f))\leq r_{\log}(X_s,\sF)
$$. Let ${x}^{\prime}$ be the unique point of $ X_s^{\prime}$ above $x\in X_s$. Since the nearby cycles functor commutes with proper push-forward, there is a Galois-equivariant isomorphism
\begin{equation*}
R\Psi_{{x}}(j_{!}\pi_{U*}\pi^{\ast}_U\sF,f) \xrightarrow{\sim}R\Psi_{{x}^{\prime}}(j_{!}^{\prime}\pi^{\ast}_U\sF,f\circ \pi)
\end{equation*}
Hence, we are left to show that
\begin{equation*}
r_{\log, K}(R^i\Psi_{{x}^{\prime}}(j_{!}^{\prime}\pi^{\ast}_U\sF,f\circ \pi))\leq r_{\log}(X_s,\sF).
\end{equation*}
Observe that the map induced by $\pi$ at the level of the henselianization of $X$ and $X'$ at the generic points of $X_s$ and $X_s^{\prime}$ is unramified. Hence, $r_{\log}(X_s,\sF)=r_{\log}(X_s^{\prime},\pi^{\ast}_U\sF)$. Thus, we are left to show that
\begin{equation*}
r_{\log, K}(R^i\Psi_{{x}^{\prime}}(j_{!}^{\prime}\pi^{\ast}_U\sF,f\circ \pi))\leq r_{\log}(X_s^{\prime},\pi^{\ast}_U\sF).
\end{equation*}
Hence, we are left to prove \eqref{rnc<r} in the case where $\sF$ extends into a locally constant and constructible sheaf of $\Lambda$-modules $\sG$ on $X\setminus X_s$. We will thus make this assumption from now on.

We put $Z^{\circ}_f=Z_f\backslash X_s$. Applying the nearby cycles functor to the short exact sequence
\begin{equation*}
0\longrightarrow j_{!}\sF\longrightarrow \sG\longrightarrow   \sG|_{Z^{\circ}_f}\longrightarrow 0
\end{equation*}
produces a distinguished triangle of complexes with Galois actions
\begin{equation*}
R\Psi_{{x}}(j_{!}\sF,f)\longrightarrow R\Psi_{{x}}(\sG,f)\longrightarrow R\Psi_{{x}}(\sG|_{Z^{\circ}_f},f|_{Z_f})\longrightarrow
\end{equation*}
Apply Step 1 to the locally constant sheaf $\sG$, we are reduced to prove
\begin{equation}\label{whatislefttoprove}
r_{\log, K}(R^i\Psi_{{x}}(\sG|_{Z^{\circ}_f},f|_{Z_f}))\leq r_{\log}(X_s,\sF).
\end{equation}
for every integer $i$. We proceed by the induction on $m$. If $m=0$, there is nothing to do. Suppose that $m>0$ and suppose that \eqref{whatislefttoprove} holds when $Z_f$ has $m-1$ irreducible components. For $1\leq n\leq m$, let $Z_n$ (resp $Z^{\circ}_n$) be the component of $Z_f$ (resp. $Z^{\circ}_f$) defined by $t_n=0$.
Set $Z_{f}^{<m}:=\bigcup_{n=1}^{m-1}Z_n$ (resp. $Z_{f}^{\circ<m}:=\bigcup_{n=1}^{m-1}Z^{\circ}_n$).  Let $i_{m} : Z^{\circ}_m \longrightarrow Z^{\circ}_{f}$, $i_{<m} : Z^{\circ<m}_f \longrightarrow Z^{\circ}_{f}$ and $j_{m} : Z^{\circ}_m \setminus  Z_{f}^{{\circ}<m} \longrightarrow Z^{\circ}_{f}$ be the canonical injections. Applying the nearby cycles functor to the short exact sequence
\begin{equation*}
0\longrightarrow j_{m!}j_{m}^*( \sG|_{Z^{\circ}_{f}})\longrightarrow \sG|_{Z^{\circ}_f}\longrightarrow   i_{<m*}(\sG|_{Z_{f}^{{\circ}<m} })\longrightarrow 0
\end{equation*}
produces a distinguished triangle of Galois modules
\begin{equation*}
R\Psi_{{x}}(j_{m!}j_{m}^*( \sG|_{Z^{\circ}_{f}}),f|_{Z_{f}})\longrightarrow R\Psi_{{x}}(\sG|_{Z^{\circ}_f},f|_{Z_f})\longrightarrow R\Psi_{{x}}(\sG|_{Z_{f}^{\circ<m} },f|_{Z_{f}^{<m} }) \longrightarrow
\end{equation*}
 By the induction hypothesis, we have
\begin{equation*}
r_{\log, K}(R^i\Psi_{{x}}(\sG|_{Z_{f}^{\circ<m} },f|_{Z_{f}^{<m} }))\leq r_{\log}(X_s,\sF)
\end{equation*}
for every integer $i$. We are thus left to prove that
\begin{equation*}
R^i\Psi_{{x}}(j_{m!}j_{m}^*( \sG|_{Z^{\circ}_{f}}),f|_{Z_{f}})\leq r_{\log}(X_s,\sF)
\end{equation*}
for every integer $i$. Let $\jmath_{m} : Z^{\circ}_m \setminus  Z_{f}^{\circ<m} \longrightarrow Z^{\circ}_m$ be the canonical injection. Observe that
\begin{equation*}
j_{m!}j_{m}^*(\sG|_{Z^{\circ}_{f}})\simeq i_{m\ast}\jmath_{m!}\jmath_{m}^*(\sG|_{Z^{\circ}_m}).
\end{equation*}
 Hence,
\begin{equation*}
R\Psi(j_{m!}j_{m}^*(\sG|_{Z^{\circ}_{f}}),f|_{Z_{f}})\simeq   R\Psi(\jmath_{m!}\jmath_{m}^*(\sG|_{Z^{\circ}_m}), f|_{Z_m})
\end{equation*}
We are thus left to prove that
\begin{equation}\label{lastthing}
r_{\log, K}(R^i\Psi_{{x}}(\jmath_{m!}\jmath_{m}^*(\sG|_{Z^{\circ}_m}), f|_{Z_m}))\leq r_{\log}(X_s,\sF).
\end{equation}
for every integer $i$. Put $T_f:=Z_m \cap Z_{f}^{<m}$, $T^{\circ}_f=Z^{\circ}_m\cap Z_{f}^{\circ<m}$ and $T:=Z_{m,s}\cup T_f$. Note that $(Z_m, T)$ is a semi-stable pair over $S$ and that $T_f$ has $m-1$ irreducible components. The sheaf $\jmath_{m}^*(\sG|_{Z^{\circ}_m})$ on $Z^{\circ}_m \setminus  Z_{f}^{\circ<m}$ is extended to a locally constant sheaf $\sG|_{Z^{\circ}_m}$.

Applying the nearby cycle functor to the short exact sequence
\begin{equation*}
0\longrightarrow \jmath_{m!}\jmath_{m}^*(\sG|_{Z^{\circ}_m})\longrightarrow \sG|_{Z^{\circ}_m}\longrightarrow   \sG|_{T^{\circ}_f}\longrightarrow 0
\end{equation*}
produces a distinguished triangle of Galois modules
\begin{equation*}
R\Psi_{{x}}(\jmath_{m!}\jmath_{m}^*(\sG|_{Z^{\circ}_m}),f|_{Z_m})\longrightarrow R\Psi_{{x}}(\sG|_{Z^{\circ}_m},f|_{Z_m})\longrightarrow R\Psi_{{x}}(\sG|_{T^{\circ}_f},f|_{T_f})\longrightarrow
\end{equation*}
 From Step 1 applied to the locally constant sheaf $\sG|_{Z^{\circ}_m}$ on the generic fiber of $Z_m$, we have for every integer $i$
\begin{equation*}
r_{\log, K}(R^i\Psi_{{x}}(\sG|_{Z^{\circ}_m},f|_{Z_m}))\leq r_{\log}(Z_{m,s},\sG|_{Z^{\circ}_m})
\end{equation*}
Applying proposition \ref{slopehd} to the sheaf $\sG$ on $X\backslash X_s$ and the closed immersion $Z_m\longrightarrow X$, we have $r_{\log}(Z_{m,s},\sG|_{Z^{\circ}_m})\leq  r_{\log}(X_s,\sG)=r_{\log}(X_s,\sF)$. Hence, for each integer $i$,
\begin{equation*}
r_{\log, K}(R^i\Psi_{{x}}(\sG|_{Z^{\circ}_m},f|_{Z_m}))\leq r_{\log}(X_s,\sF)
\end{equation*}
By the induction hypothesis and proposition \ref{slopehd}, we have for every integer $i$
\begin{equation*}
r_{\log, K}(R^i\Psi_{{x}}(\sG|_{T^{\circ}_f},f|_{T_f}))\leq r_{\log}(Z_{m,s},\sG|_{Z^{\circ}_m}) \leq r_{\log}(X_s,\sF)
\end{equation*}
Hence, the inequality \eqref{lastthing} holds and this finishes the proof of Theorem \ref{genleal}.
\end{proof}

In view of Theorem \ref{genleal}, it is natural to formulate the following local version of the conjecture from \cite{Leal}.

\begin{conjecture}\label{conjLeallocal}
Let $S$ be an henselian trait with perfect residue field of characteristic $p>0$. Let $G$ be an absolute Galois group of the field of function of $S$. Let $(X,Z)$ be a semi-stable pair over $S$. Let $U:=X\setminus Z$ and $j : U \longrightarrow X$ the canonical injection. Let $\sF$ be a locally constant and constructible sheaf of $\Lambda$-modules on $U$. Suppose that $\sF$ is tamely ramified  along the horizontal part of $Z$. Let $r_{\log}(\sF)$ be the maximum of the set of logarithmic slopes of $\sF$ at  generic points of the special fiber of $X$.

Then, for every $r>r_{\log}(\sF)$, the $r$-th upper numbering ramification subgroup of $G$ acts trivially on $R^i\Psi(j_{!}\sF,f)$ for every $i\in \mathbb Z_{\geq 0}$.

\end{conjecture}

\section{Ramification along the special fiber of a relative curve over a trait}

\subsection{}\label{notations} In this section, let $k$ be a perfect field of characteristic $p>0$. Let $S$ be a smooth connected $k$-curve, $s$ a closed point of $S$, $\bar s$ a geometric point above $s$, and $\bar\eta$ a geometric generic point of the strict henselization $S_{(\bar s)}$ of $S$ at $\bar s$. Put $V=S\setminus \{s\}$. Let $f :\mathcal{C}\longrightarrow S$ be a smooth proper relative curve over $S$ with geometrically connected fibers. Put $U=f^{-1}(V)$. Consider the cartesian diagram
\begin{equation}
\xymatrix{\relax
U\ar[r]^j\ar[d]_{f_V}\ar@{}|{\Box}[rd]&\mathcal C\ar[d]^f\ar@{}|{\Box}[rd]&D\ar[l]_i\ar[d]^{f_s}\\
V\ar[r]&S&s\ar[l]}
\end{equation}
Let $\Lambda$ be a finite field of characteristic $\ell\neq p$ and $\sF$ an locally constant and constructible sheaf of $\Lambda$-modules on $U$. Let $Z(\sF)$ be the set of closed points $x\in D$ such that $\bT_x \mathcal C$ is an irreducible component of the singular support of $j_! \sF$. 

We give a boundedness result for the cardinality of $Z(\sF)$, by applying our main result Theorem \ref{genleal}.  It is an $\ell$-adic analogue of \cite[Th. 2]{teyConjThese}.

\begin{theorem}\label{bounded}
Assume that each fiber of $f:\mathcal C\to S$ has genus $g\geq 2$. 
  Let $\sF$ be a locally constant and constructible sheaf of $\Lambda$-modules on  $U$. Then, we have
\begin{equation}
(2g-2)\cdot \rk_{\Lambda}\sF+2g\cdot \rk_{\Lambda}\sF\cdot  r_{\log}(D,\sF)\geq  \sharp Z(\sF)
\end{equation}
  \end{theorem}
\begin{proof}
The Swan conductor $\sw_s(-)$ at $s$ is a well-defined function on $D^b_c(V,\Lambda)$ associating to $\sK$ the integer $\sum_{n}(-1)^n\sw_s (\mathcal{H}^n\sK)$. From Theorem \ref{genleal}, we have for every $n=0,1,2$,
 \begin{equation}
  \sw_s (R^n f_{V*}\sF)\leq \dim_{\Lambda} H^n(\mathcal{C}_{\bar\eta},\sF|_{\mathcal{C}_{\bar\eta}})\cdot  r_{\log}(D,\sF).
  \end{equation}
Note that for $n=0,2$, we have 
\begin{equation}
\dim_{\Lambda} H^n(\mathcal{C}_{\bar\eta},\sF|_{\mathcal{C}_{\bar\eta}})\leq \rk_{\Lambda}\sF.
\end{equation} 
The Euler-Poincaré characteristic formula applied to $\sF|_{\mathcal{C}_{\bar\eta}}$ gives furthermore
\begin{equation}
\dim_{\Lambda} H^1(\mathcal{C}_{\bar\eta},\sF|_{\mathcal{C}_{\bar\eta}})\leq 2g\cdot \rk_{\Lambda}\sF.
\end{equation} 
We deduce 
 \begin{equation}\label{ptitteequation}
-2g\cdot \rk_{\Lambda}\sF\cdot  r_{\log}(D,\sF)\leq\sw_s (Rf_{V*}  \sF)\leq 2\cdot \rk_{\Lambda}\sF\cdot  r_{\log}(D,\sF).
  \end{equation}
By Saito's conductor formula recalled in Theorem \ref{conductor}, we have
 \begin{align}\label{3}
(2-2g)\cdot \rk_{\Lambda}\sF + \sw_s (Rf_{V*}  \sF) = -(CC(j_!\sF), df)_{\bT^{\ast}{\mathcal{C},D}}.
\end{align}
Since $\sF$ is locally constant on $U$, we write
\begin{align}\label{theCC}
CC(j_!\sF) =\rank_{\Lambda}\sF\cdot[\bT^{\ast}_{\mathcal{C}}\mathcal{C} ]+\alpha [\bT^{\ast}_{D}\mathcal{C}]+\sum_Z \beta_Z [Z],
\end{align}
where $Z$ runs over the set of integral 2-dimensional closed conical subschemes in $\bT^{\ast}\mathcal{C}$ supported on $D$, distinct from the conormal bundle $\bT^{\ast}_{D}\mathcal{C}$ and the zero section $\bT^*_{\mathcal{C}}\mathcal{C}$. 
\\ \indent
The section $df : \mathcal{C}\longrightarrow T^{\ast}\mathcal{C}$ is a regular embedding of codimension $2$. Since $f:\mathcal{C}\longrightarrow S$ is smooth, we have $df\cap \bT^*_{\mathcal{C}}\mathcal{C}=\emptyset$. By applying Theorem \ref{conductor} to the constant sheaf $\Lambda$ on $U$, we obtain
\begin{equation}
([\bT^{\ast}_{D}\mathcal{C}],[df])_{\bT^{\ast}{\mathcal{C},D}}=(CC(j_!\Lambda), [df])_{\bT^{\ast}{\mathcal{C},D}}=2g-2.
\end{equation}
Let $Z$ be an integral 2-dimensional closed conical subscheme in $\bT^{\ast}\mathcal{C}$ supported on $D$, distinct from $\bT^{\ast}_{D}\mathcal{C}$ and $\bT^*_{\mathcal{C}}\mathcal{C}$. Notice that, for any closed point $x\in D$, the fiber $(\bT^{\ast}_{D}\mathcal{C})_x\subset \bT^*_x\mathcal{C}$ is a $1$-dimensional $k(x)$-vector space spanned by $df_x$.   Hence we have  $\bT^{\ast}_{D}\mathcal{C}\subseteq Z$, if $\dim_k(Z\times_{\bT^*\mathcal{C}} df)\geq 1$. Thus, by our assumption, $Z$ meets $df$ at most at a finite number of points. Hence, $df$  and $Z$ intersect properly. As a general fact from intersection theory \cite[7.1]{Fulton}, we have 
\begin{align}\label{0}
([Z] , [df])_{T^{\ast}{\mathcal{C},D}}\geq \sharp (Z\times_{\bT^{\ast}\mathcal{C}} df)_{\red}
\end{align}
In particular, if $Z=\bT_x^*\mathcal{C}$ for a closed point $x\in D$, we have $$([Z] , [df])_{T^{\ast}{\mathcal{C},D}}=\sharp (Z\times_{\bT^{\ast}\mathcal{C}} df)_{\red}=1$$
Hence, we have 
\begin{align}\label{ineqconductorformula}
(2g-2)\cdot \rk_{\Lambda}\sF-\sw_s (Rf_{V*}  \sF) &=(CC(j_!\sF), [df])_{\bT^{\ast}{\mathcal{C},D}} \\
&=\sum_Z \beta_Z ([Z],[df])_{\bT^*\mathcal{C},D} +(2g-2)\cdot \alpha\nonumber\\
&\geq \sum_{Z}\beta_Z\cdot\sharp(Z\times_{\bT^{\ast}\mathcal{C}} df)_{\red} +(2g-2)\cdot \alpha.\nonumber
\end{align}
Let us recall from proposition \ref{51314} that $\alpha$ and that each $\beta_Z$ are positive integers. 
We have 
\begin{equation}\label{intersectionineq}
\sum_{Z}\beta_Z\cdot\sharp(Z\times_{\bT^{\ast}\mathcal{C}} df)_{\red}\geq \sharp Z(\sF)
\end{equation}
Thus, from \eqref{ptitteequation}, \eqref{ineqconductorformula} and \eqref{intersectionineq}
\begin{align*}
(2g-2)\cdot \rk_{\Lambda}\sF+2g\cdot \rk_{\Lambda}\sF\cdot  r_{\log}(D,\sF)&\geq\\
(2g-2)\cdot \rk_{\Lambda}\sF- \sw_s (Rf_{V*}  \sF)
&\geq   \sharp Z(\sF).
\end{align*}
This finishes the proof of Theorem  \ref{bounded}.
\end{proof}

\begin{remark}
When the fibers of $f:\mathcal C\to S$ have genus $g=0$, we have $Z(\sF)=\emptyset$.  Indeed, for the case where $g=0$, we may assume that $k$ is algebraically closed. The \'etale fundamental group of $\mathcal C_{\eta}$ is isomorphic to the Galois group $\mathrm{Gal}(\overline{\eta}/\eta)$. Hence, there is a sheaf of $\Lambda$-modules $\mathscr G$ on $\eta$ such that $\mathscr F|_{\mathcal C_{\eta}}\cong f^*_{\eta}\mathscr G$. By passing to the limit, after replacing $S$ by an \'etale neighborhood of $s$, we can find a locally constant and constructible sheaf of $\Lambda$-modules $\sG$ on $V$ such that $\sF\cong f_V^*\sG$. From \cite[Th. 1.4]{bei}, we see that 
\begin{equation*}
SS(j_!\sF)=\bT^*_{\mathcal C}\mathcal C\bigcup\bT^*_D\mathcal C.
\end{equation*}

When the fibers of $f:\mathcal C\to S$ have genus $g=1$, we refer to the more general Theorem \ref{mainprop} in the next section due to T. Saito. It is valid for abelian schemes over curves.
\end{remark}

\section{Ramification along the special fiber of an abelian scheme over a trait (after T. Saito)}

\subsection{}
In this section, $K$ denotes a field of characteristic $p\geq 0$, $A$ an abelian variety over $K$ with origin $e_A$ and $g:X\longrightarrow A$ a finite \'etale covering of degree $m$ such that $g^{-1}(e_A)$ contains a rational point $e_X\in X(K)$. Notice that $X$ has a $K$-abelian variety structure with origin $e_X$ making $g:X\longrightarrow A$ an \'etale isogeny (cf. \cite[\S 18]{Mum}).

\begin{lemma}\label{lemma1}
Assume that the \'etale isogeny $g:X\longrightarrow A$ is Galois. Then, $\ker g$ is a constant finite abelian $K$-group scheme and there is an isomorphism of abelian groups
\begin{align}\label{keraut}
\psi:(\ker g)(K) \longrightarrow \aut (X/A),\ \ \  x\mapsto +x:X\longrightarrow X \nonumber.
\end{align}
\end{lemma}

\begin{proof}
Indeed, $\ker g=g^{-1}(e_A)$ is a finite  \'etale $K$-scheme over $e_A$. Since $\ker g$ contains a rational point and is an $\mathrm{Aut}(X/A)$-torsor, it is a disjoint union of rational points in $X(K)$. Hence, $\ker g$ is a constant finite abelian $K$-group scheme. We have a bijection $$\varphi:\mathrm{Aut}(X/A)\longrightarrow \ker(g)(K),\ \ \ \sigma\mapsto \sigma(0).$$
For any $x\in (\ker g)(K)$, we have the following commutative diagram
\begin{equation*}
\xymatrix{\relax
\{x\}\times_kX\ar[dr]_g\ar@{^(->}[r]&g^{-1}(e_A)\times_kX\ar[d]^{g\times g}\ar[r]\ar@{}|-{\Box}[rd]&X\times_kX\ar[d]^{g\times g}\ar[r]^-{{m}_X}&X\ar[d]^g\\
&\{e_A\}\times_kA\ar[r]&A\times_kA\ar[r]^-{{m}_A}&A}
\end{equation*}
The outline of the diagram gives the following commutative diagram
\begin{equation*}
\xymatrix{\relax
X\ar[rd]_g\ar[r]^{+x}&X\ar[d]^g\\
&A}
\end{equation*}
Hence, $+x$ is an element of $\aut (X/A)$ for any $x\in\ker(g)(K)$. It is easy to see that $\psi$ is a groups homomorphism and is the inverse of $\varphi$. We finish the proof.
\end{proof}

\begin{lemma}\label{lemma2}
Assume that the \'etale isogeny $g:X\longrightarrow A$ is Galois. Let $n$ be a positive integer such that $(n,m)=1$. Then, for any $a\in A[n](K)$, we have 
\begin{equation*}
g^{-1}(a)\cong\coprod_{m}\spec(K).
\end{equation*}
\end{lemma}
\begin{proof}
Let $\langle a\rangle$ be the constant $K$-subgroup scheme of $A$ generated by $a\in A[n](K)$. Its cardinality is divisible by a power of $n$. We have an exact sequence of finite \'etale abelian $K$-groups
\begin{equation}\label{kergg-1aa}
0\longrightarrow\ker g\longrightarrow g^{-1}\langle a\rangle\longrightarrow \langle a\rangle\longrightarrow 0
\end{equation}
By Lemma \ref{lemma1}, $\ker g$ is a constant finite abelian $K$-group scheme of cardinality $m$. The cardinalities of $\ker g$ and $\langle a\rangle$ are co-prime, hence $\mathrm{Ext}^1(\langle a\rangle,\ker g)=0$. Thus, the short exact sequence \eqref{kergg-1aa} is split and $g^{-1}\langle a\rangle$ is a constant finite abelian $K$-group scheme. Hence $g^{-1}(a)\cong\coprod_{m}\spec(K)$. 
\end{proof}

\begin{lemma}\label{coro3}
Assume that the \'etale isogeny $g:X\longrightarrow A$ is Galois. Then, for any $a\in A(K)$ and for any $b\in (g^{-1}(a))(K)$, the following  diagram of schemes is cartesian
\begin{equation}\label{ba}
\xymatrix{\relax
X\ar[d]_g\ar[r]^{+b}\ar@{}|-{\Box}[rd]&X\ar[d]^g\\
A\ar[r]_{+a}&A}
\end{equation}
\end{lemma}
\begin{proof}

We have the following commutative diagram
\begin{equation*}
\xymatrix{\relax
\{b\}\times_kX\ar[dr]_g\ar@{^(->}[r]&g^{-1}(a)\times_kX\ar[d]^{g\times g}\ar[r]\ar@{}|-{\Box}[rd]&X\times_kX\ar[d]^{g\times g}\ar[r]^-{{m}_X}&X\ar[d]^g\\
&\{a\}\times_kA\ar[r]&A\times_kA\ar[r]^-{{m}_A}&A}
\end{equation*}
The outline of the diagram above implies that \eqref{ba} is commutative. Since $+b$ and $+a$ are isomorphisms, the diagram \eqref{ba} is furthermore cartesian.

\end{proof}

\begin{corollary}\label{coro4}
Assume that the \'etale isogeny $g:X\longrightarrow A$ is Galois. Let $\Lambda$ be a commutative finite ring in which $p$ is invertible. Let $\sF$ be a locally constant and constructible sheaf of $\Lambda$-modules on $A$ trivialized by $X$. Then, for any $a\in A[n](K)$ with $(n,m)=1$, we have 
\begin{equation*}
(+a)^*\sF\cong\sF.
\end{equation*}
\end{corollary}
\begin{proof}
By lemma \ref{lemma2}, we know that $g^{-1}(a)$ is a disjoint union of $m$ copies of $\spec(K)$. Let $b$ be a rational point in $g^{-1}(a)$. By lemma \ref{coro3}, 
the pull-back by $+a:A\longrightarrow A$ induces an automorphism 
\begin{equation*}
(+a)^\sharp:\aut(X/A)\longrightarrow\aut (X/A),\ \ \ \sigma\mapsto (-b)\circ \sigma\circ (+b).
\end{equation*}
Let $\rho_{\sF}:\aut(X/A)\longrightarrow \mathrm{GL}(\Lambda^r)$ be the representation corresponding to $\sF$. Then, $(+a)^*\sF$ corresponds to the representation $\rho_{\sF}\circ(+a)^{\sharp}$.
By lemma \ref{lemma1}, $\sigma=+x$ for some $x\in g^{-1}(e_A)\subset X(K)$. Since $(+a)^\sharp(+x)= (-b)\circ (+x)\circ (+b)=+x$, we deduce that $(+a)^\sharp$ is the identity. Thus, we have $(+a)^*\sF\cong\sF$.
\end{proof}

\begin{lemma}\label{lemma5}
Let $\Lambda$ be a commutative finite ring in which $p$ is invertible. Let $\sF$ be a locally constant and constructible sheaf of $\Lambda$-modules on $A$. Then, there exists a Galois \'etale covering $h:Y\longrightarrow A$ trivializing $\sF$ such that
\begin{itemize}
\item[(i)]
There is a Galois extension $K'/K$ such that $h:Y\longrightarrow A$ is the  composition of a map $h':Y\longrightarrow A_{K'}$ with the canonical projection $\pi':A_{K'}\longrightarrow A$;
\item[(ii)]
$h':Y\longrightarrow A_{K'}$ is a finite Galois \'etale isogeny of abelian varieties over $K'$.
\end{itemize}
\end{lemma}
\begin{proof}
Let $Z\longrightarrow A$ be a finite Galois \'etale covering trivializing 
$\sF$. Let $K'$ be a finite Galois extension of $K$ such that the image of $Z_{K'}(K')\longrightarrow A_{K'}(K')$ contains the origin $e_{A_{K'}}$ of $A_{K'}$. Let $Y$ be a connected component of $Z_{K'}$ containing  a pre-image $e^{\prime}$ of $e_{A_{K'}}$ in $Z_{K'}(K')$. The induced morphism $h:Y\longrightarrow A$ satisfies the conditions of lemma \ref{lemma5}.
\end{proof}

\begin{proposition}\label{invariant}
Let $\Lambda$ be a commutative finite ring in which $p$ is invertible. Let $\sF$ be a locally constant and constructible sheaf of $\Lambda$-modules on $A$. Then, there exists a positive integer $m(\sF)$ depending on $\sF$ such that, for any positive integer $n$ co-prime to $m(\sF)$ and for any $a\in A[n](K)$, we have $(+a)^*\sF\cong \sF$.
\end{proposition}

\begin{proof}
Let $h:Y\xrightarrow{h'} A_{K'}\xrightarrow{\pi'}A$ be a finite Galois \'etale covering trivializing $\sF$ as  in lemma \ref{lemma5}.  We put $m'=[Y:A_{K'}]$. We denote by $m''$ the cardinality of $\aut(\aut(Y/A))$. We put $m(\sF)=m'm''$. Let  $n$ be a positive integer co-prime to $m(\sF)$. Pick $a\in A[n](K)$. We abusively denote by $a$ its pull-back in $A_{K'}[n](K')$. Since $n$ is co-prime to $m'$, lemma \ref{lemma2} ensures that $h'^{-1}(a)$ is a disjoint union of copies of $\spec(K')$. Pick $b\in h'^{-1}(a)$. From lemma \ref{coro3}, we have the following cartesian diagrams
\begin{equation*}
\xymatrix{\relax
Y\ar[d]_{h'}\ar[r]^-{+b}\ar@{}|-{\Box}[rd]&Y\ar[d]^{h'}\\
A_{K'}\ar[d]_{\pi'}\ar[r]^-{+a}\ar@{}|-{\Box}[rd]&A_{K'}\ar[d]^{\pi'}\\
A\ar[r]^-{+a}&A}
\end{equation*} 
The pull-back by $+a:A\longrightarrow A$ induces an automorphism 
\begin{equation*}
(+a)^\sharp:\aut(Y/A)\longrightarrow\aut (Y/A),\ \ \ \sigma\mapsto (-b)\circ \sigma\circ (+b).
\end{equation*}
On the one hand, since $na=e_A$, we have $nb\in (\ker h')(K')$. From lemma \ref{lemma1}, we deduce
\begin{equation*}
+nb\in \aut(Y/A_{K'})\subset\aut(Y/A).
\end{equation*}
 Hence
\begin{equation*}
\left((+a)^\sharp\right)^n=\mathrm{Ad}(+nb)\in\mathrm{Inn}(\aut(Y/A)).
\end{equation*}
On the other hand, since $(+a)^\sharp\in \aut(\aut(Y/A))$, we get 
\begin{equation*}
\left((+a)^{\sharp}\right)^{m''}=\mathrm{id}_{\aut(Y/A)}.
\end{equation*}
 Since $(n,m'')=1$, we deduce that 
\begin{equation*}
(+a)^\sharp \in \mathrm{Inn}(\aut(Y/A)).
\end{equation*}
Hence, the representations of $\aut(Y/A)$ induced by $\sF$ and that of $(+a)^*\sF$ are isomorphic. Thus, we have $(+a)^*\sF\cong \sF$.
\end{proof}

\subsection{}\label{dense}
Let $S$ be a strict henselian trait with an algebraically closed residue field of charateristic $p>0$. We denote by $s$ its closed point, by $\eta$ its generic point and by $\overline\eta$ a geometric point above $\eta$.  Let $\cA$ be an abelian scheme over $S$ of relative dimension $g>0$. Let $n$ be a positive integer co-prime to $p$. The $S$-group scheme $\cA[n]$ is finite \'etale over $S$. Since $S$ is a strictly henselian trait and since $\cA_s[n]$ is a constant group scheme over $s$ of group $(\mathbb Z/n\mathbb Z)^{2g}$, the $S$-group scheme $\cA[n]$ is a constant group scheme over $S$ of group $(\mathbb Z/n\mathbb Z)^{2g}$. Restricting to the generic fiber, $\cA_\eta[n]$ is a constant group scheme of $(\mathbb Z/n\mathbb Z)^{2g}$ over $\eta$. 
There are canonical bijections between the three sets $\cA_{\eta}[n](\eta)$,  $\cA_{s}[n](s)$ and $\cA[n](S)$. Notice that the set of closed points 
\begin{equation*}
\bigcup_{i\geq 1}\cA_{\overline\eta}[n^i](\overline\eta)\subset \cA_{\overline\eta}(\overline\eta)
\end{equation*}
is dense in $\cA_{\overline\eta}$. Hence, the union of rational points 
\begin{equation*}
\bigcup_{i\geq 1}\cA_{\eta}[n^i](\eta)\subset \cA_{\eta}(\eta)
\end{equation*}
is also dense in $\cA_{\eta}$.

\subsection{}\label{settingACF}
Let $k$ be an algebraically closed field of characteristic $p>0$, $C$ a connected smooth $k$-curve, $c$ a closed point of $C$, $\jmath:V=C-\{c\}\longrightarrow C$ the canonical injection. Let $S$ be the henselization of $C$ at $c$ and let $\eta$ be the generic point of $S$. Let $\cA$ be an abelian scheme over $C$ of relative dimension $g>0$. We have the following commutative diagrams
\begin{equation*}
\xymatrix{\relax
U\ar[d]_{f_V}\ar[r]^j\ar@{}|{\Box}[rd]&\cA\ar[d]^{f}\ar@{}|{\Box}[rd]&\cA_c\ar[l]\ar[d]^{f_c}& &\cA_\eta\ar[d]\ar[r]\ar@{}|{\Box}[rd]&\cA_S\ar[d]\ar@{}|{\Box}[rd]&\cA_{ c}\ar[d]\ar[l]\\
V\ar[r]_{\jmath}&C&c\ar[l]& &\eta\ar[r]&S& c\ar[l]}
\end{equation*}
Let $n$ be a positive integer co-prime to $p$. For an element $a\in \cA_c[n](c)$, we denote by $a_S$ the corresponding element in 
$\cA_S[n](S)$.  Then, there exists an \'etale neighborhood $\pi:C'\longrightarrow C$ of $c$ such that 
The $S$-point $a_S:S\longrightarrow \cA_S$  descends to a $C^{\prime}$-point $a_{C'}:C'\longrightarrow \cA_{C'}$. 
Notice that $\bT^*\cA_{C'}=\bT^*\cA\times_\cA\cA_{C'}$. The $C'$-point $a_{C'}$ induces 
an isomorphism 
$$d(+a):\bT^*\cA\times_{\cA}\cA_c=\bT^*\cA_{C'}\times_{\cA_{C'}}\cA_c\xrightarrow{d(+a_{C'})} \bT^*\cA_{C'}\times_{\cA_{C'}}\cA_c=\bT^*\cA\times_{\cA}\cA_c,$$ 
which is independent of the choice of $a_{C'}$.
For any closed point $x\in \cA_c$, it induces an isomorphism of fibers $d(+a)_x:\bT^*_x\cA\to \bT^*_{x-a}\cA$.

Let $\Lambda$ be a finite field of characteristic $\ell\neq p$ and let $\sF$ be a non-trivial locally constant and constructible sheaf of $\Lambda$-modules on $U$.

\begin{theorem}\label{mainprop}
We use the notations and assumptions of \ref{settingACF}. Then, 
\begin{itemize}
\item[(i)]
The ramification of $\sF$ along the divisor $\cA_c$ of $\cA$ is non-degenerate. In particular, for any closed point $x\in \cA_c$, the fiber of the singular support  $SS(j_!\sF)\longrightarrow \mathcal A$ above $x$ is a union of lines in $\mathbb \bT^*_x\mathcal A$. 
\item[(ii)]
For a closed point $x\in\cA_c$, we put $E_x=SS(j_!\sF)\times_{\cA}x$. Then, there exists a positive integer $r$ depending on $\sF$ such that, for any closed point $x\in\cA_c$ and any $a\in \cA_c[n](c)$ with $(n,rp)=1$, we have 
\begin{equation}\label{d+a=x-a}
d(+a)_x(E_x)=E_{x-a}
\end{equation}
In particular, assuming that $f:\cA\longrightarrow C$ is a trivial family of abelian varieties, i.e., $\cA\cong \cA_c\times_k C$, we have 
\begin{equation}\label{SS=E*Ac}
SS(j_!\sF)=\bT^*_{\cA}\cA\bigcup (E_0\times_k\cA_c),
\end{equation}
where $0$ denotes the origin of $\cA_c$ and $E_0\times_k\cA_c$ is considered as a closed conoical subset of $\bT^*\cA\times_\cA\cA_c\subseteq \bT^*\cA$ by the isomorphism $\bT^*\cA\times_\cA\cA_c\cong \bT^*_0\cA\times_k\cA_c$.
\end{itemize}
\end{theorem}

\begin{proof}
By proposition \ref{invariant}, we can find an integer $r$ depending on $\sF|_{\cA_\eta}$, such that, for any positive integer $n$ co-prime to $r$ and any rational point $a_\eta\in \cA_{\eta}[n](\eta)$, there is an isomorphism $(+a_\eta)^*(\sF|_{\cA_\eta})\cong\sF|_{\cA_\eta}$.  Pick $a_{\eta}\in \bigcup_{(n,rp)=1}\cA_{\eta}[n](\eta)$ and denote by $a_S$ and $a$ the corresponding elements in $\bigcup_{(n,rp)=1}\cA_{S}[n^i](S)$ and $\bigcup_{(n,rp)=1}\cA_{c}[n^i](c)$ respectively (\ref{dense}). Then, there exists an \'etale neighborhood $\pi:C'\longrightarrow C$ of $c$ such that 
\begin{itemize}
\item[(1)]
The $S$-point $a_S:S\longrightarrow \cA_S$  descends to a $C^{\prime}$-point 
\begin{equation}
a_{C'}:C'\longrightarrow \cA_{C'};
\end{equation}
\item[(2)]
The isomorphism $(+a_\eta)^*(\sF|_{\cA_\eta})\cong \sF|_{\cA_\eta}$  descends to an isomorphism 
\begin{equation}\label{+aC}
(+a_{C'})^*(j_!\sF|_{\cA_{C'}})\cong j_!\sF|_{\cA_{C'}}
\end{equation}
\end{itemize}

(i). Let $W$ be a dense open subset of $\cA_c$ along which the ramification of $\sF$  is non-degenerate (\cite[Lemma 3.2]{wr}). Since the non-degeneracy property is \'etale local, the non-degenerate loci of $\sF|_{\cA_{V'}}$, where $V'=C'-\{c\}$, and 
$\sF$ along $\cA_c$ are the same. Hence, the isomorphism \eqref{+aC} implies that the ramification of $\sF$ is also non-degenerate along $W+a\subset \cA_c$. Since $\mathcal X=\bigcup_{(n,rp)=1}\cA_{c}[n](c)$ is  dense in $\cA_c$ (\ref{dense}), we have 
\begin{equation*}
\cA_c=\bigcup_{a\in \mathcal X}W+a.
\end{equation*}
Hence, the ramification of $\sF$ along $\cA_c$ is non-degenerate. We finish the proof of (i).

(ii). The isomorphism \eqref{+aC} implies that 
\begin{align}
(+a_{C'})^\circ(SS(j_!\sF|_{\cA_{C'}}))&=SS(j_!\sF|_{\cA_{C'}})\label{+a^circSS=SS}\\
d(+a)(SS(j_!\sF|_{\cA_{C'}})\times_{\cA}\cA_c)&=SS(j_!\sF|_{\cA_{C'}})\times_{\cA}\cA_c\nonumber
\end{align}
 Taking the fiber of $SS(j_!\sF)\longrightarrow\cA$ at a closed point $x\in\cA_c$, \eqref{+a^circSS=SS} implies that
 \begin{equation}
d(+a)_x(E_x)=E_{x-a},
\end{equation}
for any $a\in\mathcal X$. The equality \eqref{d+a=x-a} is proved.

 Assume that $\cA\cong \cA_c\times_kC$ is a trivial family. Since $SS(j_!\sF)$ is a closed conical subset of $\bT^*\cA$, for any $x\in\cA_c(c)$, we have 
 \begin{equation}
\ol{\bigcup_{a\in \mathcal X}d(+a)_x(E_x)}=\ol{\bigcup_{a\in \mathcal X}E_{x-a}}\subseteq SS(j_!\sF).
 \end{equation}
 For any $x\in \cA_c(c)$, we denote by $x_C\in\cA(C)$ its lift by the base change. Since $\mathcal X$ is dense in $\cA_c$, we have
 \begin{align*}
 d(+x_C)_0(E_0)\subseteq \ol{\bigcup_{a\in \mathcal X}d(+a)_0(E_0)}\subseteq SS(j_!\sF)\\
 d(-x_C)_0(E_x)\subseteq \ol{\bigcup_{a\in \mathcal X}d(+a)_x(E_x)}\subseteq SS(j_!\sF)
 \end{align*}
 for any $x\in \cA_c(c)$. Hence, 
 \begin{equation}
 d(+x_C)_0(E_0)\subseteq E_{-x},\ \ \ \textrm{and}\ \ \ d(-x_C)_0(E_x)\subseteq E_{0},
 \end{equation}
 for any $x\in \cA_c(c)$. Therefore, $d(+x_C)_0(E_0)=E_{-x}$ for any $x\in \cA_c(c)$. Hence, we have 
 \begin{equation}
 \ol{\bigcup_{x\in \cA_c(c)}d(+x_C)_0(E_0)}=\ol{\bigcup_{x\in \cA_c(c)}E_{-x}}=SS(j_!\sF)\times_{\cA}\cA_c
 \end{equation}
Notice that $\bT^*\cA\times_\cA\cA_c\cong \bT^*_0\cA\times_k \cA_c$ is a trivial vector bundle. We have 
\begin{equation}
SS(j_!\sF)\times_{\cA}\cA_c= \ol{\bigcup_{x\in \cA_c(c)}d(+x_C)_0(E_0)}=E_0\times_k\cA_c.
 \end{equation}
 We obtain \eqref{SS=E*Ac}.
   \end{proof}

\begin{corollary}\label{l-adicNP}
We use the notations and assumptions from \ref{settingACF}. Let  $x$ be a closed point in $\cA_c$. Let $T$ be a 
smooth $k$-curve  in $\cA$ meeting the divisor $\cA_c$ transversely at $x$ and such that the immersion $T\longrightarrow \cA$ is $SS(j_!\sF)$-transversal at $x$. Then, the slope decomposition of $\sF|_T$ at $x$ and that of $\sF$ at the generic point of $\cA_c$ have the same Newton polygon (\ref{invM}) .
\end{corollary}
\begin{proof}
From Theorem \ref{mainprop}, the ramification of $\sF$ along $\cA_c$ is non-degenerate around $x$.  By \cite[Lemma 3.1]{wr}, after replacing $\cA$ by a sufficiently small \'etale neighborhood $U$ of $x$ in $\cA$, the sheaf $\sF$ decomposed as a direct sum of locally constant and constructible sheaves $\{\sF_i\}_{i\in I}$ on $U$ such that the ramification of each $\sF_i$ along $\cA_c$ is non-degenerate and isoclinic. Notice that $T\longrightarrow \cA$ is also $SS(j_!\sF_i)$-transversal for each $i\in I$. By \cite[Prop. 2.22]{wr}, the ramification of $\sF_i|_T$ at $x$ is also isoclinic at $x$ with the same slope as $\sF_i$. Hence, the slope decomposition of $\sF|_T$ at $x$ and the slope decomposition of $\sF$ at the generic point of $\cA_c$ have the same Newton polygon. 
\end{proof}

In \cite{Tsuzuki}, Tsuzuki showed the constancy of Newton polygons for convergent $F$-isocrystals on  abelian varieties over a finite field. Corollary \ref{l-adicNP} offers a similar constancy of Newton polygons for the ramification of locally constant and constructible \'etale sheaves on abelian schemes over a curve  ramified along a vertical divisor.

\begin{proposition}\label{saitoapplication}
We use the notations and assumptions of \ref{settingACF}. Then, 
\begin{itemize}
\item[(i)]
For any irreducible component $T$ of $SS(j_!\sF)$ different from $\mathbb T^*_{\cA_c}\cA$, we have 
$$T\bigcap \mathbb T^*_{\cA_c}\cA\subset \mathbb T^*_{\cA}\cA.$$
\item[(ii)]
We have the following conductor formula 
\begin{equation}\label{condform=0}
\sum_{i=0}^{2g}(-1)^i\mathrm{sw}_c(R^if_{V*}\sF)=0.
\end{equation}
\end{itemize}
\end{proposition}
\begin{proof}
(i) Let $B$ be  the union of the irreducible components of $SS(j_!\sF)$ different from $\mathbb T^*_{\cA_c}\cA$ and $\bT^*_{\cA}\cA$. Hence, either we have $SS(j_!\sF)\times_\cA\cA_c=B$ or $SS(j_!\sF)\times_\cA\cA_c=B\bigcup\bT^*_{\cA_c}\cA$. It is sufficient to show that $B\bigcap \mathbb T^*_{\cA_c}\cA\subset \mathbb T^*_{\cA}\cA$.  By proposition \ref{invariant}, we can find an integer $r$ depending on $\sF|_{\cA_\eta}$, such that, for any positive integer $n$ co-prime to $r$ and any rational point $a\in \cA_{\eta}[n](\eta)$, there is an isomorphism $(+a)^*(\sF|_{\cA_\eta})\cong\sF|_{\cA_\eta}$.  We take an element $a$ in $\mathcal X=\bigcup_{(n,rp)=1}\cA_c[n](c)$. By the proof of proposition \ref{mainprop}, we have 
\begin{equation}\label{d+aSS=SS}
d(+a)(SS(j_!\sF|_{\cA_{C'}})\times_{\cA}\cA_c)=SS(j_!\sF|_{\cA_{C'}})\times_{\cA}\cA_c.
\end{equation}  
We have the following commutative diagram of vector bundles on $\cA_c$
\begin{equation}\label{bun+a}
\xymatrix{\relax
0\ar[r]&\bT^*_{\cA_c}\cA\ar[d]^{\phi}\ar[r]&\bT^*\cA\times_{\cA}\cA_c\ar[r]\ar[d]^{d(+a)}&\bT^*\cA_c\ar[r]\ar[d]^{d(+a)}&0\\
0\ar[r]&\bT^*_{\cA_c}\cA\ar[r]&\bT^*\cA\times_{\cA}\cA_c\ar[r]&\bT^*\cA_c\ar[r]&0}
\end{equation}
where the horizontal lines are exact sequences and $\phi$ is the restriction of $d(+a):\bT^*\cA\times_{\cA}\cA_c\longrightarrow \bT^*\cA\times_{\cA}\cA_c$ to $\bT^*_{\cA_c}\cA$, which is also an isomorphism. Hence 
$$
d(+a)(\bT^*_{\cA_c}\cA)=\bT^*_{\cA_c}\cA.
$$
Thus, \eqref{d+aSS=SS} implies that $d(+a)(B)=B$.  It is valid for any $a\in\mathcal X$. Suppose that the intersection of  $B$ and $\bT^*_{\cA_c}\cA$ is not contained in $\bT^*_{\cA}\cA$, i.e., there exists a closed point $x\in\cA_c$, such that $\mathbb T^*_{\cA,x}\cA\subseteq B_x$.  Thus, for any $a\in \mathcal X$, we have 
 $$\bT^*_{\cA_c,x-a}\cA=d(+a)(\bT^*_{\cA_c,x}\cA)\subset d(+a)(B)= B.$$ 
 Since $\mathcal X=\bigcup_{(n,rp)=1} \cA_c[n](c)$ is dense in $\cA_c$, we have 
$$
\bT^*_{\cA_c}\cA=\ol{\bigcup_{a\in \mathcal X} \bT^*_{\cA_c,x-a}\cA}.
$$ 
Hence, $\bT^*_{\cA_c}\cA\subseteq B$ which contradicts the fact that each irreducible component of $B$ is different from $\bT^*_{\cA_c}\cA$. This finishes the proof of proposition \ref{saitoapplication} $(i)$.\\ \indent
We now prove \ref{saitoapplication} $(ii)$. By Saito's conductor formula recalled in Theorem \ref{conductor}, we have
 \begin{align}\label{3}
\sum_{i=0}^{2g}(-1)^i\mathrm{sw}_c(R^if_{V*}\sF)= -(CC(j_!\sF), [df])_{\bT^{\ast}{\mathcal{A},\cA_c}}
\end{align}
where $df$ denotes a section of $\bT^*\cA$  as recalled in \ref{def df}.
Since $f:\cA\longrightarrow S$ is smooth, we have $df\bigcap \bT^*_{\cA}\cA=\emptyset$.
By applying \eqref{3} to the constant sheaf $\Lambda$ on $U$, we obtain that
\begin{equation*}
([\bT^{\ast}_{\cA_c}\mathcal{A}],[df])_{\bT^{\ast}{\mathcal{A},\cA_c}}=(-1)^{g+1}(CC(j_!\Lambda), [df])_{\bT^{\ast}{\mathcal{A},\cA_c}}=0.
\end{equation*}
Note that the image of the canonical map $\bT^*C\times_C\cA_c\to \bT^*\cA\times_{\cA}\cA_c$ is the conormal bundle $\bT^*_{\cA_c}\cA$. From $(i)$, for any irreducible component $T$ of $SS(j_!\sF)$ different from $\bT^*_{\cA_c}\cA$, we have $df\bigcap T=\emptyset$. In summary, we have $(CC(j_!\sF), [df])_{\bT^{\ast}{\mathcal{A},\cA_c}}=0$. From \eqref{3}, we deduce the sought-after vanishing \eqref{condform=0}.
\end{proof}

\subsection{}
In an earlier version of this article, the vanishing property of the conductor \eqref{condform=0} was proved in the relative dimension $1$ case by using Theorem \ref{genleal}. T. Saito then pointed out the remarkable fact \ref{invariant} to the authors. By using it, we could obtain the non-degeneracy of the ramification of \'etale sheaves on abelian schemes along a vertical divisor (Theorem \ref{mainprop}) as well as  \eqref{condform=0} without any restriction on the relative dimension.



\end{document}